\documentclass[reqno,10pt,a4paper]{amsart}

\usepackage{hyperref}
\usepackage{latexsym,amsmath, bm}
\usepackage{enumerate}
\usepackage{amsfonts}
\usepackage{amssymb}
\usepackage{geometry}
\usepackage{latexsym}
\usepackage{fixmath}
\usepackage{faktor}
\usepackage{mathtools}
\usepackage[T1]{fontenc}
\usepackage{fourier}
\usepackage{bbm}
\usepackage{color}
\usepackage{pgfplots}
\usepackage{graphicx}
\usepackage{fullpage}

\numberwithin{equation}{section}

\newcommand{\ER}{\Erdos-\Renyi}

\newcommand{\vH}{\vec H}

\newcommand{\vJ}{\vec J}

\newcommand{\vR}{\vec R}

\newcommand\fr{\mathfrak{r}}

\newcommand\fA{\mathfrak{A}}

\newcommand\fx{\mathfrak{x}}

\newcommand\cA{\mathcal{A}}
\newcommand\cB{\mathcal{B}}
\newcommand\cC{\mathcal{C}}

\newcommand\cE{\mathcal{E}}

\newcommand\cG{\mathcal{G}}
\newcommand\cH{\mathcal{H}}

\newcommand\cM{\mathcal{M}}

\newcommand\cO{\mathcal{O}}
\newcommand\cP{\mathcal{P}}

\newcommand\cR{\mathcal{R}}
\newcommand\cS{\mathcal{S}}
\newcommand\cT{\mathcal{T}}

\newcommand\cX{\mathcal{X}}
\newcommand\cY{\mathcal{Y}}
\newcommand\cZ{\mathcal{Z}}

\newcommand\MU{\vec\mu}

\newcommand\RHO{{\vec\rho}}

\newcommand{\TT}{\mathbb T}
\newcommand{\GG}{\mathbb G}
\newcommand{\G}{\mathbf G}

\newcommand\GER{\GG_{\text{ER}}}

\newcommand\dTV{d_{\mathrm{TV}}}

\newcommand{\Po}{{\rm Po}}

\renewcommand{\vec}[1]{\boldsymbol{#1}}
\newcommand{\vecone}{\vec{1}}

\newcommand\KL[2]{D_{\mathrm{KL}}\bc{{{#1}\|{#2}}}}

\newcommand\eul{\mathrm{e}}
\newcommand\eps{\varepsilon}

\newcommand\NN{\mathbb{N}}
\newcommand\pr{\mathbb{P}} 
\renewcommand\Pr{\pr}

\newcommand\Erw{\mathbb{E}}
\newcommand\ex{\mathbb{E}}

\newcommand\RR{\mathbb{R}}

\newcommand{\whp}{w.h.p.}

\newcommand{\tensor}{\otimes}

\newcommand\id{\mathrm{id}}

\newcommand{\ignore}[1]{\relax}

\newtheorem{definition}{Definition}[section]

\newtheorem{theorem}[definition]{Theorem}
\newtheorem{lemma}[definition]{Lemma}
\newtheorem{proposition}[definition]{Proposition}
\newtheorem{corollary}[definition]{Corollary}

\newtheorem{fact}[definition]{Fact}

\newcommand\Lem{Lemma}
\newcommand\Prop{Proposition}
\newcommand\Thm{Theorem}

\newcommand\Cor{Corollary}
\newcommand\Sec{Section}

\newcommand\bc[1]{\left({#1}\right)}
\newcommand\cbc[1]{\left\{{#1}\right\}}
\newcommand\bcfr[2]{\bc{\frac{#1}{#2}}}

\newcommand\brk[1]{\left\lbrack{#1}\right\rbrack}

\newcommand\abs[1]{\left|{#1}\right|}

\newcommand{\Erdos}{Erd\H{o}s}
\newcommand{\Renyi}{R\'enyi}

\newcommand{\Mezard}{M\'ezard}

\newcommand{\ecutm}{\mbox{\sc MaxCut}(\GG)}

\newcommand{\limb}{\lim_{\b \to \infty}}

\newcommand{\pd}[1]{\frac{\partial}{\partial #1}}

\renewcommand{\a}{\alpha}
\renewcommand{\b}{\beta}
\newcommand\eb{\eul^{-\b}}
\newcommand{\ebb}{\eul^{-2\b}}
\newcommand{\bstar}{\b^*}
\newcommand{\bdag}{\b^\dag}

\newcommand{\KLmrra}{\KL{\mu_\a}{\rho_\a \tensor \rho_\a}}
\newcommand{\cAda}{\cA}
\newcommand{\limn}{\lim_{n \to \infty}}

\newcommand{\zAx}{\zeta\cAda^d\xi}

\newcommand{\ZbG}{Z_{\GG,\b}}

\newcommand{\pmone}{\cbc{\pm 1}}
\newcommand{\vwe}{\{v,w\}}

\begin{document}
	
\title{The Ising antiferromagnet and max cut on random regular graphs}

\thanks{Amin Coja-Oghlan and Philipp Loick are supported by DFG CO 646/3}

\author{Amin Coja-Oghlan, Philipp Loick, Bal\'azs F. Mezei, Gregory B. Sorkin}

\address{Amin Coja-Oghlan, {\tt acoghlan@math.uni-frankfurt.de}, Goethe University, Mathematics Institute, 10 Robert Mayer St, Frankfurt 60325, Germany.}

\address{Philipp Loick, {\tt loick@math.uni-frankfurt.de}, Goethe University, Mathematics Institute, 10 Robert Mayer St, Frankfurt 60325, Germany.}

\address{Bal\'azs F. Mezei, {\tt balazs.mezei@cs.ox.ac.uk}, Oxford University, Wolfson Building, Parks Road, Oxford OX1 3QD, UK}

\address{Gregory B. Sorkin, {\tt g.b.sorkin@lse.ac.uk}, London School of Economics, Department of Mathematics, Houghton St, London WC2A 2AE, UK}

\begin{abstract}
	The Ising antiferromagnet is an important statistical physics model with close connections to the {\sc Max Cut} problem.
	Combining spatial mixing arguments with the method of moments and the interpolation method, we pinpoint the replica symmetry breaking phase transition predicted by physicists.
	Additionally, we rigorously establish upper bounds on the {\sc Max Cut} of random regular graphs predicted by Zdeborov\'a and Boettcher [Journal of Statistical Mechanics 2010].
	As an application we prove that the information-theoretic threshold of the disassortative stochastic block model on random regular graphs coincides with the Kesten-Stigum bound.
	\hfill{\em MSc:~05C80.}
\end{abstract}

\maketitle

\section{Introduction}\label{sec_intro}

\subsection{Motivation}
The Ising model is to statistical physics what the $k$-SAT problem is to computer science or the Ramsey problem to combinatorics:
it serves as a benchmark for new techniques to prove their mettle.
Devised by Lenz in the 1920s to explain magnetism, the Ising model can be defined on an arbitrary graph $G$.
Think of the vertices of $G$ as iron atoms that each carry one of two possible magnetic spins, $\pm1$.
With the topology of interactions defined by the edges of $G$, the Hamiltonian $\cH_G$ (the `energy' function) maps a spin configuration $\sigma\in\{\pm1\}^V$ to the number of edges of $G$ that link two vertices with the same spin, i.e.,
\begin{align}\label{eqHam}
	\cH_G(\sigma)&=\sum_{\vwe\in E}\frac{1+\sigma_v\sigma_w}{2}.
\end{align}
Together with a real parameter $\beta$ the Hamiltonian induces a probability distribution $\mu_{G,\beta}$ on the set of spin configurations via
\begin{align}\label{eq_boltzmann}
	\mu_{G,\beta}(\sigma)&=\frac{\exp(-\beta\cH_G(\sigma))}{Z_{G,\beta}}\qquad(\sigma\in\{\pm1\}^V)&\mbox{where}&&
	Z_{G,\beta}&=\sum_{\tau\in\{\pm1\}^V}\exp(-\beta\cH_G(\tau)).
\end{align}
This probability measure is called the {Boltzmann distribution}.
The normalising term $Z_{G,\beta}$ is known as the {partition function}.
If $\beta>0$, then $\mu_{G,\beta}$ favours spin configurations $\sigma$ with a small number of edges joining vertices with the same spin; this case is known as the {antiferromagnetic} Ising model.
By contrast, in the {ferromagnetic} case $\beta<0$ configurations with many aligned spins receive a boost. 

Both variants of the Ising model are of keen interest in physics and the literature on each, rigorous as well as non-rigorous, is vast~\cite{Friedli_2017,Huang_2009}.
But the antiferromagnetic Ising model appears to be more challenging.
According to physics lore this is because its Boltzmann distribution is prone to a complicated type of long-range correlation known as `replica symmetry breaking'.
Another way to see the challenge is that from the partition function we could solve the NP-complete problem {\sc Max Cut}:
as $\beta$ increases the mass of the Boltzmann distribution shifts to spin configurations with more edges joining vertices with opposite spins. 
Ultimately the measure concentrates on the maximum cuts of the graph $G$,
and it is well known (and easy to check) that
\begin{align}\label{eqMaxCut}
\mbox{\sc MaxCut}(G)&=\frac{dn}{2}+\lim_{\beta\to\infty}\frac{\partial}{\partial\beta}\log Z_{G,\beta}.
\end{align}

We study the Ising antiferromagnet on the random $d$-regular graph $\GG=\GG(n,d)$.
From a statistical physics perspective, this example has been suggested as one of the simplest models where replica symmetry breaking is expected to occur.
Fond of lattice-like geometries, physicists favour the random regular graph, which converges to the $d$-regular tree in the Benjamini-Schramm topology, over the \ER\ model.
In particular, regularity greatly simplifies the physics `cavity equations' that Zdeborov\'a and Boettcher \cite{Zdeborova_2009} employed to put forward a beautiful, well-known conjecture about {\sc Max Cut} on random regular graphs.
From a combinatorics perspective, the random regular graph provides a neat but notoriously challenging model for {\sc Max Cut}, both structurally (determining the fraction of edges it should be possible to cut, asymptotically almost surely) and algorithmically (finding algorithms that give large cuts in such graphs). 
The problem has received a great deal of attention in the combinatorics community, e.g.~\cite{Coppersmith_2004,Csoka_2016,Diaz_2003,Diaz_2007, Kalapala_2002,Kardos_2012,Sorkin_2019}.
Additionally, the Ising model is intimately related to the regular version of the disassortative stochastic block model~\cite{Coja_2020}, a prominent case study in Bayesian inference.

\subsection{Our contributions and paper outline}
Our first contribution is to identify the precise value of $\beta$ where the replica symmetry breaking phase transition occurs;
see Theorem~\ref{thm_elogz}.
A common approach to problems of this type would be the trick of bounding the second moment of random regular graphs by that of the \Erdos-\Renyi\, as applied in \cite{Achlioptas_2004, Coja_2016}. Since this approach fails in our case, we instead turn to harnessing spatial mixing arguments to establish the phase transition.
As a ramification of the replica symmetry breaking phase transition, our second contribution is to derive the information-theoretic threshold of the disassortative regular stochastic block model;
see Theorem~\ref{Thm_sbm}.
Our third contribution is to establish rigorously the upper bound on the {\sc Max Cut} of the random regular graph predicted by Zdeborov\'a and Boettcher;
see Corollary~\ref{cor_max_cut}.
Specifically, their prediction is based on the so-called `1-step replica symmetry breaking' formalism from physics.
Using the interpolation method it is easy to obtain a rigorous upper bound that comes as a variational problem.
However, this variational problem appears rather unwieldy at first glance.
But by expressing the variational problem as a certain random walk that we can analyze, we obtain an elegant explicit expression whose numerical results match those of Zdeborov\'a and Boettcher.

In the remainder of Section~\ref{sec_intro} we state the main results of the paper precisely.
In Section~\ref{sec_techniques} we outline the proof strategy, and in Section~\ref{sec_discussion} we discuss the advances over earlier work.
Details of the proofs of Theorem~\ref{thm_elogz}, Theorem~\ref{Thm_1rsb}, and Corollary~\ref{cor_max_cut} are given respectively in Sections \ref{sec_thm_1}, \ref{sec_thm_2}, and \ref{sec_cor}.

\subsection{Replica symmetry breaking}
The key quantity associated with the Ising model on $\GG$ is the partition function $Z_{\GG,\beta}$.
This is because various combinatorially meaningful observables derive from the partition function via differentiation; see, for example, \eqref{eqMaxCut}.
Because $Z_{\GG,\beta}$ scales exponentially in $n$, it is common to consider the normalised logarithm $n^{-1}\log Z_{\GG,\beta}$, known as the free energy.
Routine arguments show that this random variable concentrates about its mean.
Hence, we are led to investigate the function
\begin{align}\label{eqPhi}
	\Phi_d:\b\in(0,\infty)\mapsto\limn \frac{1}{n} \Erw \brk{\log \ZbG};
\end{align}
the limit is known to exist for all $d\geq3,\beta>0$~\cite{Bayati_2010}.
In particular, for a fixed $d\geq3$ the singularities $\beta$ of $\Phi_d$ (the points at which $\Phi_d$ cannot be expanded to an absolutely convergent power series) are called the {\em phase transitions} of the Ising model.
Hence, from a mathematical physics point of view computing $\Phi_d$ and pinpointing the phase transitions is the key challenge associated with the model.

Jensen's inequality immediately yields the inequality
\begin{align}\label{eqJensen}
\Phi_d(\beta)&\leq\limn \frac{1}{n}\log\ex\brk{\ZbG}=\log2+\frac{d}{2}\log\frac{1+\eul^{-\beta}}{2} ,
\end{align}
where the equality is taken from the (easy) calculation of
$\Erw[\ZbG]$ as \eqref{eqTech8} in \Lem~\ref{Cor_simple*}.
A tempting first guess might be that \eqref{eqJensen} is generally tight.
Combinatorially this would indicate that the Boltzmann distribution $\mu_{\GG,\beta}$ is free from long-range correlations.
To see this, consider the experiment of removing a single random edge $\vec e=\vwe$ from $\GG$.
Because short cycles are scarce, \whp\ the vertices $v,w$ have distance $\Omega(\log n)$ in $\GG-\vec e$.
Hence, in the absence of long-range correlations, in a sample $\vec\sigma$ from the Boltzmann distribution of $\GG-\vec e$, the spins $\vec\sigma_v,\vec\sigma_w$ should be asymptotically independent, i.e., $\pr[\vec\sigma_v=\vec\sigma_w\mid\GG,\vec e]=1/2+o(1)$.
Therefore, adding $\vec e$ back in should change the partition function by
\begin{align*}
	\log\frac{Z_{\GG,\b}}{Z_{\GG-\vec e,\b}}&=\log\bc{1-(1-\eul^{-\b})\pr[\vec\sigma_v=\vec\sigma_w\mid\GG,\vec e]}\sim\log\frac{1+\eul^{-\b}}{2}.
\end{align*}
Removing a random edge $dn/2$ times until all the edges are gone and observing that the partition function of the empty graph equals $2^n$, we would thus obtain equality in \eqref{eqJensen}.
However, the following theorem shows that \eqref{eqJensen} is tight 
only for $\beta$ up to an explicit threshold $\bstar$.

\begin{theorem} \label{thm_elogz}
For any $d\geq3$ let
\begin{align}\label{eqthm_elogz1}
    \bstar(d) = \log \bc{\frac{\sqrt{d-1}+1}{\sqrt{d-1}-1}} .
\end{align}
\begin{enumerate}[(i)]
	\item If $\beta<\bstar(d)$, then $$\Phi_d(\b) = \log2+\frac{d}{2}\log\frac{1+\eul^{-\beta}}{2}.$$
	\item If $\beta>\bstar(d)$, then $$\Phi_d(\b) < \log2+\frac{d}{2}\log\frac{1+\eul^{-\beta}}{2}.$$ 
\end{enumerate}
\end{theorem}

Because the function $\beta\mapsto\log2+\frac{d}{2}\log\frac{1+\eul^{-\beta}}{2}$ is analytic, \Thm~\ref{thm_elogz} implies that $\Phi_d(\beta)$ is non-analytic at the point $\beta=\bstar$.
Hence, there occurs a phase transition at $\bstar$ that separates a regime where $\ZbG$ concentrates about its mean from a regime where the mean is driven up by rare events.
In physics jargon this phase transition is called the {\em replica symmetry breaking} transition.
The value $\bstar$ has a special combinatorial meaning:
it is the reconstruction threshold for a broadcasting process first studied by Kesten and Stigum~\cite{Kesten_1966},
and is thus known as the `Kesten-Stigum bound'.
Thus, \Thm~\ref{thm_elogz} shows that the replica symmetry breaking phase transition in the Ising antiferromagnet on $\GG$ occurs precisely at the Kesten-Stigum bound.

\subsection{Bounding {\sc Max Cut}}\label{Sec_MC}
\Thm~\ref{thm_elogz} does not provide a simple expression for $\Phi_d(\b)$ for $\b>\bstar$.
Indeed, such a simple expression may not exist.
This is because according to physics predictions the value $\Phi_d(\b)$ for $\b>\bstar$ results from a complicated variational problem over an infinite-dimensional space of probability measures that meticulously characterises the long-range correlations of the Boltzmann distribution~\cite{Coja_2019}.

Yet in the limit $\beta\to\infty$ it is possible to derive an explicit upper bound on the value of $\Phi_d(\b)$.
To state this bound consider the following right stochastic band matrix $\cM$ of size $(d+1) \times (d+1)$:
\begin{align}\label{eq_Mdef}
    \cM =
    \begin{bmatrix}
        0  & 1  & 0 & \cdots & \cdots & \cdots & 0 \\
        \frac12  & 0 & \frac12  & \ddots & & & \vdots \\
        0 & \frac12  & 0 & \frac12 & \ddots & & \vdots \\
        \vdots & \ddots & \ddots & \ddots & \ddots & \ddots & \vdots  \\
        \vdots & & \ddots & \ddots & \ddots & \ddots & 0 \\
        \vdots & & & \ddots & \frac12 & 0 & \frac12 \\
        0 & \cdots & \cdots & \cdots & 0 & 1 & 0
    \end{bmatrix}.
\end{align}
Moreover, let
\begin{align}
F_d(\alpha,z)&=-\frac{\log \bc{\zAx}}{\log z} + \frac{d\log \bc{1-2\a^2+2\a^2z}}{2\log z},\qquad\mbox{where}\label{eq_F_def}\\
\cAda&=(1-2\a)\id+2\a \sqrt z\cM,\label{eq_A_def}\\
\zeta&= {\begin{bmatrix} 1, & 0, & 0, & \cdots & \end{bmatrix}}\in\RR^{1\times(d+1)} \label{eq_v_def}, \\
\xi &= {\begin{bmatrix} 1, & z^{-1/2}, & z^{-1}, & z^{-3/2}, & \cdots \end{bmatrix}}^T\in\RR^{(d+1)\times 1} \label{eq_w_def}.
\end{align}

\begin{theorem}\label{Thm_1rsb}
For any $d\geq3$ we have
\begin{align*}
\lim_{\beta\to\infty}\b^{-1}\Phi_d(\b)&\le\inf_{\substack{0<\a\leq1/2\\0<z<1}}F_d(\alpha,z).
\end{align*}
\end{theorem}

Since \eqref{eqMaxCut} shows that the {\sc Max Cut} problem is tied to $\Phi_d(\b)$ for large $\b$, we can use \Thm~\ref{Thm_1rsb} to derive upper bounds on the maximum cut size of the random regular graph.
\begin{corollary} \label{cor_max_cut}
Let $\ecutm$ be the number of edges cut by a maximum cut of $\GG$.
Then, \whp,
\begin{align*}
\ecutm\leq\frac{dn}{2}\inf_{\substack{0<\a<1/2\\0<z<1}}\bc{1+\frac{2}{d}F_d(\alpha,z)}+o(n).
\end{align*}
\end{corollary}

Zdeborov\'a and Boettcher \cite{Zdeborova_2009} conjectured that the expected maximum cut size in a random regular graph is upper bounded by the solution to the one-step replica-symmetry breaking equations and provided numerical estimates of the resulting cut size.
Corollary~\ref{cor_max_cut} matches their numbers.

Table~\ref{Tab_mc} displays the upper bounds from Corollary~\ref{cor_max_cut} for $d=3,\ldots,10$.
For comparison the table also contains the previous best rigorous upper bounds we are aware of, and the best rigorous lower bounds.
Upper bounds appear to have received little attention.
For $d>3$ the upper bounds shown come from straightforward application of the first moment method, counting cuts of the given size; this can be done either by standard counting arguments or using \cite[Corollary 2.8]{Coja_2016}, in either case followed by a small numerical computation.
For $d=3$ better upper bounds come from the first moment method but restricting to cuts satisfying some local maximality conditions; the bound shown is from \cite{Sorkin_2019}.
The lower bounds result from analyses of algorithms, and are from \cite{Gamarnik_2018} via \cite{Csoka_2015} for $d=3$, \cite{Diaz_2003} for $d=4$ and \cite{Diaz_2007} for $d>4$.

\begin{table}[ht]
\label{table_bound}
\begin{center}
\begin{tabular}{ |c|c|c|c|c|c|c|c|c| }
 \hline
 $d$ & 3 & 4 & 5 & 6 & 7 & 8 & 9 & 10 \\
 \hline
 \mbox{best previous upper bound} & 0.9320 & 0.8900 & 0.8539 & 0.8260 & 0.8038 & 0.7855 & 0.7701 & 0.7570 \\
 \Cor~\ref{cor_max_cut} upper bound & 0.9241 & 0.8683 & 0.8350 & 0.8049 & 0.7851 & 0.7659 & 0.7523 & 0.7388 \\
\mbox{best lower bound} & 0.9067 & 0.8333 & 0.7989 & 0.7775 & 0.7571 & 0.7404 & 0.7263 & 0.7144 \\
\mbox{expected cut size at $\bstar$} & 0.8536 & 0.7887 & 0.7500 & 0.7236 & 0.7041 & 0.6890 & 0.6768 & 0.6667 \\
\mbox{expected cut size at Gibbs uniqueness} & 0.7500 & 0.6667 & 0.6250 & 0.6000 & 0.5833 & 0.5714 & 0.5625 & 0.5556 \\
\hline
\end{tabular}
\end{center}
\caption{Bounds on the fraction of edges in a maximum cut of $\GG(n,d)$.
}
\label{Tab_mc}
\end{table}

The article of Zdeborov\'a and Boettcher contains a second, more prominent conjecture that ties together the {\sc Min Bisection} and {\sc Max Cut} problems on random regular graphs,
namely that the two cases result \whp\ in asymptotically equal numbers of
edges `dissatisfied' (respectively, cut and not cut).
Unfortunately the methods of the present work do not appear to shed light on this question.

However, the work does shed light on a different question of interest:
as an application of \Thm~\ref{thm_elogz} we can calculate the information-theoretic threshold of the disassortative stochastic block model.

\subsection{The stochastic block model}\label{Sec_sbm}
Over the past decade the stochastic block model has become a prominent benchmark for Bayesian inference as well as graph clustering.
The impressive literature on the model is surveyed in~\cite{Abbe_2017, Moore_2017}.
Like the Ising model, the stochastic block model comes in two variants.
In the assortative version edges are more likely join vertices with the same spin while in the disassortative model edges are more likely to occur between vertices with opposite spins.
Thus, the disassortative variant resembles the Ising antiferromagnet.

Formally the $d$-regular disassortative stochastic block model is defined by way of the following experiment.
Let $V_n=\{v_1,\ldots,v_n\}$ be a set of $n$ vertices.
In a first step we draw a spin assignment $\vec\sigma^*\in\{\pm1\}^{V_n}$ uniformly at random.
Subsequently we draw a $d$-regular graph $\GG^*=\GG^*(\vec\sigma^*)$ from the distribution
\begin{align}\label{eqsbm}
\pr\brk{\GG^*=G\mid\vec\sigma^*}&\propto\exp(-\beta\cH_G(\vec\sigma^*)).
\end{align}
Thus, the probability that a given $d$-regular graph $G$ comes up is proportional to the Boltzmann weight $\exp(-\beta\cH_G(\vec\sigma^*))$ of the `ground truth' $\vec\sigma^*$. 

The obvious question is whether the bias introduced by \eqref{eqsbm} has a discernible impact on the distribution of the graph.
In other words, is it possible to tell $\GG^*$ apart from the `null model' $\GG$?
To formalise this we use the Kullback-Leibler divergence of $\GG^*$ from $\GG$, 
\begin{align*}
	\KL{\GG^*}{\GG}&=\sum_G\pr\brk{\GG^*=G}\log\frac{\pr\brk{\GG^*=G}}{\pr\brk{\GG=G}}.
\end{align*}
The Kullback-Leibler divergence is an information-theoretic potential that gauges the difference between random objects.
Specifically, if $\KL{\GG^*}{\GG}=o(n)$ then extensive observables such as the maximum cut value or the logarithm of the partition function in the two random graph models are asymptotically equal
\cite{MMbook}.
By contrast, if $\KL{\GG^*}{\GG}=\Omega(n)$, then one can tell the two random graph models apart by calculating the partition function
\cite{CKPZ}.
In particular, in the latter case there exists a (not necessarily efficient) algorithm $\mathtt A$ that given a graph $G$ outputs $\mathtt A(G)\in\{0,1\}$ such that
\begin{align}\label{eqA}
\lim_{n\to\infty}\pr\brk{\mathtt A(\GG)=0}&=\lim_{n\to\infty}\pr\brk{\mathtt A(\GG^*)=1}=1.
\end{align}
Hence, $\mathtt A$, the essence of which is calculating the partition function, distinguishes the stochastic block model from the null model with high probability. 

\begin{theorem}\label{Thm_sbm}
For any $d\geq3$ the following are true.
\begin{enumerate}[(i)]
	\item If $\beta<\bstar(d)$, then $\limn \KL{\GG^*}{\GG}/n=0$ and $\limn \KL{\GG}{\GG^*}/n=0$.
	\item If $\beta>\bstar(d)$, then $\limn \KL{\GG^*}{\GG}/n>0$ and $\limn \KL{\GG}{\GG^*}/n>0$. 
\end{enumerate}
\end{theorem}
By definition, then, $\bstar$ is the information-theoretic threshold of the stochastic block model.

\section{Techniques} \label{sec_techniques}

\noindent
This section contains a survey of the proofs of the main results and the techniques they are based on.
We begin with the proof of the first part of \Thm~\ref{thm_elogz}, which combines moment computations with a spatial mixing argument.
To motivate this combination we first discuss the \ER\ case, in which a straightforward moment calculation does the trick. Subsequently we discuss the proof of the second part of \Thm~\ref{thm_elogz}, which relies on the connection between the Ising model and the stochastic block model.
This connection also shows how \Thm~\ref{Thm_sbm} follows from \Thm~\ref{thm_elogz}.
The final subsection then deals with the the proof of \Thm~\ref{Thm_1rsb}, based on the interpolation method.

\subsection{The second moment method}

To get started, we will compute the typical value of the Ising partition function using the method of moments for the \Erdos-\Renyi~model.
To this end, we reproduce the calculation 
by Mossel, Neeman and Sly~\cite{Mossel_2015} for the \ER\ model $\GER$
where $m=dn/2$ edges are drawn uniformly at random.
(We skip their supplementing of the second moment method with small subgraph conditioning for increased precision.)
We will show why this does not directly extend to the random regular model.

For the first moment we simply obtain
\begin{align}\label{eqTech1}
\ex\brk{Z_{\GER,\b}}&=\sum_{\sigma\in\{\pm1\}^{V_n}}\Erw\brk{\exp(-\beta\cH_{\GER}(\sigma))}
=\sum_{\sigma\in\{\pm1\}^{V_n}}\bc{1-\frac{1-\eul^{-\beta}}{n^2}\sum_{i,j=1}^n\vecone\{\sigma_{v_i}=\sigma_{v_j}\}}^{m+o(n)}.
\end{align}
The second equality holds because the edges of $\GER$ are asymptotically independent.
A moment's reflection reveals that the expression in the braces is maximised by $\sigma$ such that $\sum_{i}\sigma_{v_i}=o(n)$.
Combinatorially this means that $\sigma$ corresponds to an approximately balanced cut.
Since there are $2^{n+o(n)}$ such $\sigma$, \eqref{eqTech1} yields
\begin{align}\label{eqTech2}
\ex\brk{Z_{\GER,\b}}&=2^{n+o(n)}\bcfr{1+\eul^{-\b}}{2}^{dn/2}
 =\exp\bc{n\bc{\Big(1-\frac{d}{2}\Big)\log(2)+\frac{d}{2}\log(1+\eul^{-\b})+o(1)}}.
\end{align}

Calculating the second moment is similarly straightforward.
Indeed, we obtain
\begin{align}\label{eqTech3}
\ex&\brk{Z_{\GER,\b}^2}=\sum_{\sigma,\sigma'\in\{\pm1\}^{V_n}}\Erw\brk{\exp(-\beta\cH_{\GER}(\sigma)-\beta\cH_{\GER}(\sigma'))}\\
&=\sum_{\sigma,\sigma'}\bc{1-\frac{1}{n^2}\sum_{i,j=1}^n\bc{1-\eul^{-\beta}}\bc{\vecone\{\sigma_{v_i}=\sigma_{v_j}\}+\vecone\{\sigma'_{v_i}=\sigma'_{v_j}\}}
			-\bc{1-\eul^{-\b}}^2\vecone\{\sigma_{v_i}=\sigma_{v_j}\wedge\sigma'_{v_i}=\sigma'_{v_j}\}}^{m+o(n)}.\nonumber
\end{align}
As in the first moment calculation it is easy to see that asymptotically balanced $\sigma,\sigma'$ dominate.
Moreover, rearranging the sum according to the inner product $a=\sigma\cdot\sigma'$, we obtain
\begin{align}
\ex\brk{Z_{\GER,\b}^2}
&=\sum_{a=-n}^n\binom{n}{(n-a)/4,(n-a)/4,(n+a)/4,(n+a)/4}\bc{\frac{(1+\eul^{-\b})^2}{4}+\bcfr an^2\frac{(1-\eul^{-\b})^2}4}^{m+o(n)}.\label{eqTech4}
\end{align}
Introducing $\alpha=a/n$ and the entropy function $H(p)=-p\log p-(1-p)\log(1-p)$ for $0<p<1$, we can apply Stirling's formula to simplify \eqref{eqTech4} to
\begin{align}\label{eqTech5}
\ex\brk{Z_{\GER,\b}^2}&=\exp\bc{n\max_{-1<\alpha<1}f_d(\alpha,\beta)+o(n)},\qquad\mbox{where}\\
f_{d}(\alpha,\beta)&=(1-d)\log(2)+H((1+\alpha)/2)+\frac{d}{2}\log\bc{(1+\eul^{-\b})^2+\alpha^2(1-\eul^{-\b})^2}.\label{eqTech5a}
\end{align}

Substituting $\alpha=0$ into \eqref{eqTech5a} yields
\begin{align}\label{eqTech5_2}
f_{d}(0,\beta)&=(2-d)\log(2)+d\log\bc{1+\eul^{-\b}}
\end{align}
which is twice the exponent from \eqref{eqTech2}.
Hence, if $f_d(\a,\b)$ attains its maximum at $\alpha=0$, then \eqref{eqTech1} and \eqref{eqTech5} show that $\ex[Z_{\GER,\b}^2]/\ex[Z_{\GER,\b}]^2=\exp(o(n))$.
Routine concentration arguments therefore apply and show that $Z_{\GER,\b}$ concentrates about its expectation.
In particular, we obtain 
\begin{align}\label{eqTech6}
	\lim_{n\to\infty}\frac{1}{n}\Erw\brk{\log Z_{\GER,\b}}&=\lim_{n\to\infty}\frac{1}{n}\log \Erw\brk{Z_{\GER,\b}}=\log2+\frac{d}{2}\log\frac{1+\eul^{-\b}}{2}\qquad\mbox{if }\max_{-1<\alpha<1}f_d(\alpha,\beta)=f_d(0,\b).
\end{align}
By contrast, if the maximum in \eqref{eqTech5} is attained at $\alpha\neq0$, then the second moment exceeds the square of the first moment exponentially.
Hence, the moment method succeeds iff $f_d(\alpha,\beta)$ attains its maximum at $\alpha=0$. 

Whether or not this is the case depends on the value of $\beta$.
Specifically, for a given $d\geq3$ the function $f_d(\alpha,\beta)$ is maximised at $\alpha=0$, and \eqref{eqTech6} is satisfied, if $$\beta\leq\bdag(d) =  \log{\frac{\sqrt{d}+1}{\sqrt{d}-1}};$$
as mentioned above, this discovery belongs to Mossel, Neeman and Sly~\cite{Mossel_2015}.
Note that $\bstar(d)=\bdag(d-1) >\bdag(d)$.
Conversely, results from~\cite{Coja_2018,Mossel_2015} imply that
\begin{align}\label{eqTech7}
\lim_{n\to\infty}\frac{1}{n}\Erw\brk{\log Z_{\GER,\b}}&<\log2+\frac{d}{2}\log\frac{1+\eul^{-\b}}{2} \qquad \mbox{for $\b>\bdag(d)$}.
\end{align}
Hence, $\bdag(d)$ marks the replica symmetry breaking threshold for the Ising model on the \ER\ graph.

To what extent do these considerations carry over to the random regular graph $\GG(n,d)$?
The following lemma shows that the first moment for random regular graphs is about the same as in the \ER\ case; 
the calculations, given in~\cite{Mossel_2015}, are similar to those above.

\begin{lemma}\label{Cor_simple*}
	For any $d\geq3,\b>0$ we have 
\begin{align}\label{eqTech8}
\ex[Z_{\GG,\b}]=\Theta\bc{ 2^n\bc{ \frac{ 1+\eul^{\b} }2 }^{dn/2} }.
\end{align}
\end{lemma}
For the second moment, the expression from \eqref{eqTech5} for the \Erdos-\Renyi\ model carries over to the random regular graph and yields the upper bound
\begin{align}
\ex\brk{Z_{\GG(n,d),\b}^2}&\leq \exp\bc{n\max_{-1<\alpha<1}f_d(\alpha,\beta)+o(n)} \label{eqTech9}.
\end{align}
The fact that the bound extends to random regular graphs may not appear entirely immediate; the analytic explanation derives from the convexity of the Kullback-Leibler divergence \cite{Achlioptas_2004, Coja_2016}.
A similar trick has been applied with some success to various random regular graph problems, notably graph colouring \cite{Achlioptas_2004}.
Unfortunately, in our case the second moment trick only yields the desired solution for $\beta < \bdag(d)$,
while \Thm~\ref{thm_elogz} requires it for all $\beta < \bstar(d)$.
Indeed, for $\beta > \bdag(d)$ the trick fails in a rather spectacular way: once $\beta$ crosses above $\bdag(d)$ the value $\alpha=0$ turns from a global maximum of the function $f_d(\alpha, \beta)$ into a local minimum! Figure~\ref{fig_secondmoment} provides an illustration.
Thus, we have to turn to other means to establish the first part of \Thm~\ref{thm_elogz}, which we explore next.

\begin{figure}
\centering
\begin{minipage}{.3\textwidth}
  \centering
  \includegraphics[width=0.9\linewidth]{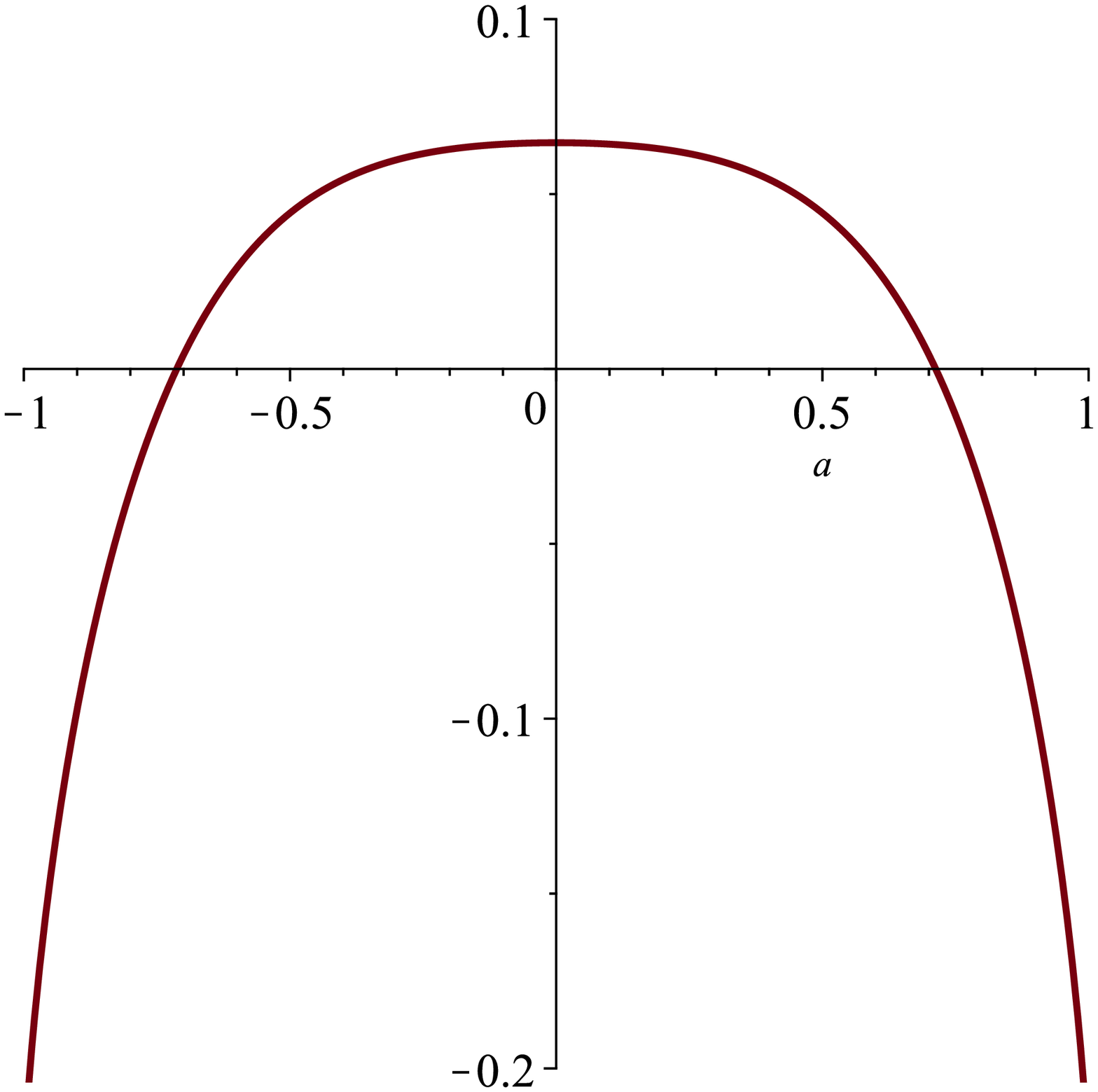}
\end{minipage}%
\begin{minipage}{.3\textwidth}
  \centering
  \includegraphics[width=0.9\linewidth]{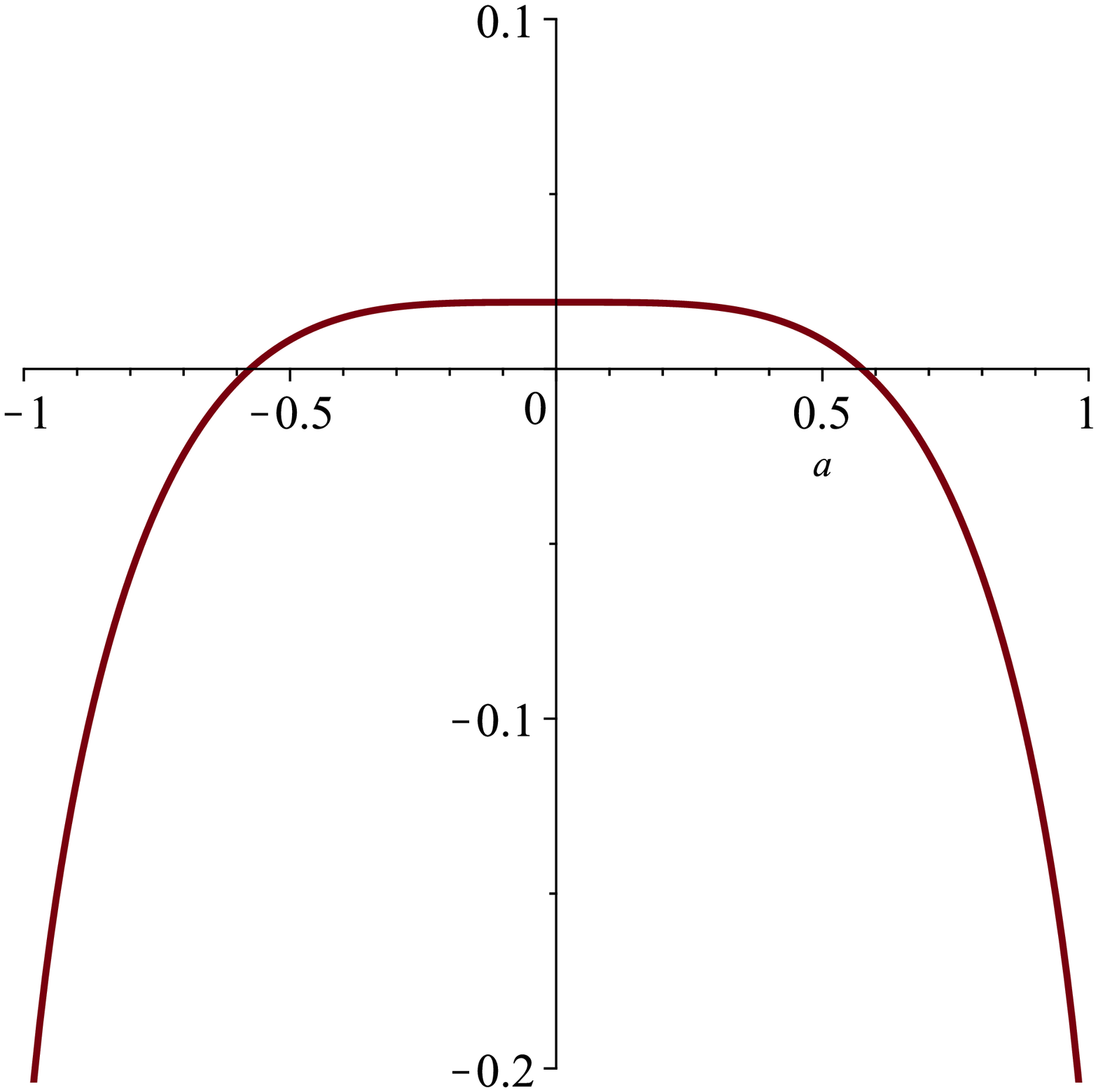}
\end{minipage}
\begin{minipage}{.3\textwidth}
  \centering
  \includegraphics[width=0.9\linewidth]{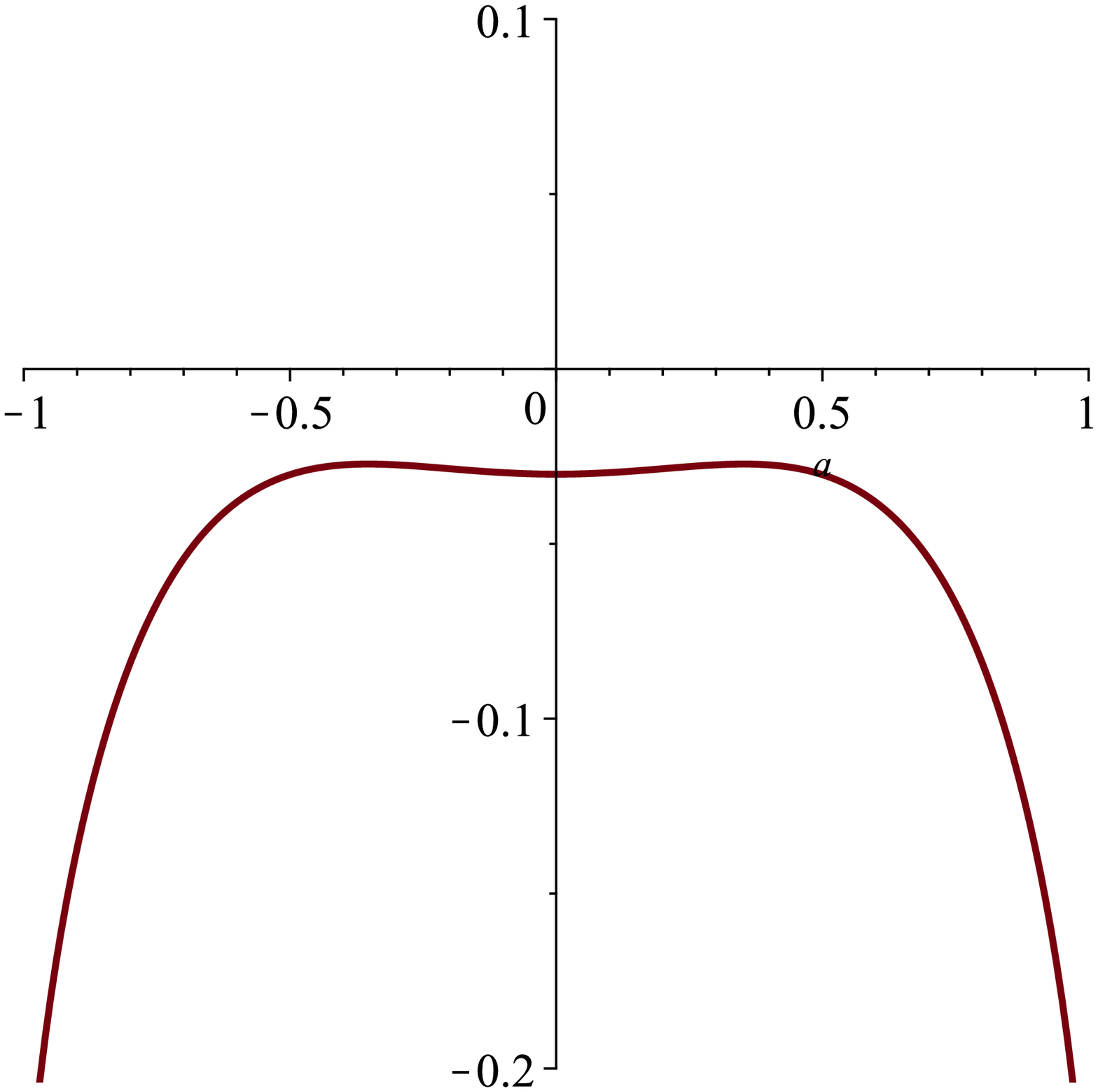}
\end{minipage}
\caption{The function $f_d(\alpha,\b)$ for $d=3$ and $\beta=1.25$ (left), $\beta=1.32$ (middle) and $\beta=1.40$ (right). For $d=3$, we have $\bdag \approx 1.32$ and $\bstar \approx 1.76$.}
\label{fig_secondmoment}
\end{figure}
 
\subsection{Broadcasting and non-reconstruction}\label{Sec_bnr}
Our approach towards proving the first part of \Thm~\ref{thm_elogz} relies on combining spatial mixing arguments with the method of moments. To be precise,
we will exhibit an event $\cO$ such that for all $\b<\b^*(d)$,
\begin{align}\label{eqO}
\ex\brk{Z_{\GG(n,d),\b}\vecone\cbc{\cO}}&=\Theta(\ex\brk{Z_{\GG(n,d),\b}})=\Theta\bc{2^n\bcfr{1+\eul^{-\b}}{2}^{dn/2}},&
\ex\brk{Z_{\GG(n,d),\b}^2\vecone\cbc{\cO}}&=4^{n+o(n)}\bcfr{1+\eul^{-\b}}{2}^{dn}.
\end{align}
Together with routine concentration arguments \eqref{eqO} will imply the first part of \Thm~\ref{thm_elogz}.

To elaborate, the event $\cO$ concerns the relative location of two typical samples from the Boltzmann distribution.
Hence, for a graph $G$ let $\vec\sigma_G,\vec\sigma'_G$ denote two independent samples from $\mu_{G,\b}$.
Then for a sequence $\eps_n=o(1)$ that tends to zero slowly enough (and that we will specify precisely in due course) we let
\begin{align}\label{eqOdef}
\cO&=\cbc{\ex\brk{|\vec\sigma_{\GG}\cdot\vec\sigma_{\GG}'| \mid\GG}<\eps_n n}.
\end{align}
Thus, $\cO$ is the event that two typical samples from the Boltzmann distribution are nearly orthogonal.
Since the combinatorial interpretation of $\alpha$ in \eqref{eqTech9} is to pinpoint the value of the inner product of spin configurations that renders the largest contribution, one might reasonably hope that conditioning on $\cO$ will eliminate the need for taking values $\alpha\neq0$ into consideration.
Indeed, the proof of the second part of \eqref{eqO} will be relatively straightforward.
Unfortunately, it turns out that the same cannot quite be said of the proof of the first part.

\begin{proposition} \label{prop_recon_energy}
The event $\cO$ from \eqref{eqOdef} satisfies \eqref{eqO} for all $d\geq3$, $\b\leq\b^*(d)$.
\end{proposition}

The proof of \Prop~\ref{prop_recon_energy} uses two tools: the stochastic block model $\GG^*$ from \eqref{eqsbm} and the analysis of a broadcasting process on the infinite $d$-regular tree from~\cite{Bleher_1995}.
Specifically, the stochastic block model will help to derive the first part of \eqref{eqO}.
Indeed, the definition \eqref{eqsbm} suggests that the probability that $\GG^*=G$ for a given graph $G$ should be roughly proportional to the partition function $Z_{G,\b}$ (see e.g. \cite{CKPZ}).
This is because $G$ has a chance proportional to 
$\sum_{\sigma\in\{\pm1\}^{V_n}} \big( \exp(-\b\cH_G(\sigma)) / \sum_G \exp(-\b\cH_G(\sigma)) \big)$,
and if the denominators were the same over all $\sigma$,
this would be proportional to $Z_{G,\b}=\sum_\sigma\exp(-\b\cH_{G}(\sigma))$.
By symmetry each denominator depends only on the magnetisation of $\sigma$ (the sum of its entries), and summands with magnetisation near 0 are far more frequent, so it is reasonable to hope that they dominate the sum.
Capitalising on this intuition, the following lemma shows that we can make use of $\GG^*$ to establish the first part of \eqref{eqO}.

\begin{lemma}\label{Lemma_hunch}
Let $d\geq3,\b>0$.
If $\pr\brk{\GG^*\in\cO}\sim1$, then $\ex\brk{Z_{\GG,\b}\vecone\cbc{\cO}}=\Theta \bc{\ex\brk{Z_{\GG,\b}}}$.
\end{lemma}

To show that $\pr\brk{\GG^*\in\cO}\sim1$ we will couple the planted model $\GG^*$ with a broadcasting process on the infinite $(d-1)$-ary tree $\TT_{d-1}$.
Let $u_0$ signify the (degree-$d$) root of $\TT_{d-1}$.
Proceeding down the tree, the broadcasting process constructs an assignment $\vec\tau\in\{\pm1\}^{V(\TT_{d-1})}$ as follows.
Initially we choose $\vec\tau_{u_0}\in\{\pm1\}$ uniformly at random.
Subsequently, having defined $\vec\tau_u$ for all $u$ at distance at most $\ell$ from $u_0$ already, we define the value $\vec\tau_w$ of a child $w$ of such a vertex $u$ by letting
\begin{align}\label{eqbroadcast}
\pr\brk{\vec\tau_w=\vec\tau_u\mid\vec\tau_u}&=\eb/(1+\eb).
\end{align}
In words, $w$ retains the spin of its parent with probability $\eb/(1+\eb)$, and is assigned the opposite spin with the remaining probability $1/(1+\eb)$.
Let $\cT_\ell$ denote the $\sigma$-algebra generated by the spins $\vec\tau_u$ of all vertices $u$ at distance greater than $\ell$ from $u_0$.
The following result shows that the spin $\tau_{v_0}$ decorrelates from $\cT_\ell$ in the limit of large $\ell$ if $\b<\b^*(d)$.
In other words, the broadcasting process `forgets' the spin of the root, a property known as non-reconstruction~\cite{Bleher_1995}.

\begin{lemma}[\cite{Bleher_1995}] \label{lem_reconstruction}
Let $d\geq3$ and $\b<\b^*(d)$.
Then 
\begin{align}\label{eqlem_reconstruction}
\lim_{\ell \to \infty}\ex\abs{\pr\brk{\vec\tau_{v_0}=1\mid\cT_\ell}-\frac{1}{2}}=0. 
\end{align}
\end{lemma}

\noindent
As an aside, $\b^*(d)$ actually is the sharp non-reconstruction threshold \cite{Higuchi_1977}.
Thus, \eqref{eqlem_reconstruction} ceases to hold for $\b>\b^*(d)$.

Equipped with \Lem~\ref{lem_reconstruction} the proof of the condition $\ex\brk{Z_{\GG,\b}\vecone\cbc{\cO}}\sim\ex\brk{Z_{\GG,\b}}$ proceeds as follows.
For a typical vertex of $\GG^*$, say $v_1$, we couple the spins that the planted configuration $\vec\sigma^*$ assigns to vertices in the $\ell$-ball around $v_1$ with the broadcasting process.
This coupling is based on the fact that the random regular graph $\GG^*$ converges to the $d$-regular tree in the Benjamini-Schramm topology.
Then we re-sample the spins inside the $\ell$-ball given the spins assigned to all the vertices at distance greater than $\ell$ from $v_1$ according to the Boltzmann distribution $\mu_{\GG^*,\b}$.
Let $\vec\sigma^{**}$ denote the resulting spin configuration.
\Lem~\ref{lem_reconstruction} will enable us to conclude that the re-sampled spin $\vec\sigma^{**}_{v_1}$ is asymptotically independent from the original spin $\vec\sigma^*_{v_1}$.
Finally, we will show that both $\vec\sigma^*$ and $\vec\sigma^{**}$ are distributed approximately as two samples from the Boltzmann distribution $\mu_{\GG^*,\b}$, thereby deriving the following.

\begin{lemma}\label{cor_reconstruction}
Let $d\geq3$ and $\b<\b^*(d)$.
Then $\pr\brk{\GG^*\in\cO}\sim1$.
\end{lemma} 

\noindent
As shown in Section~\ref{sec_thm_1},
\Prop~\ref{prop_recon_energy} will be an easy consequence of \Lem~\ref{Lemma_hunch} and \Lem~\ref{cor_reconstruction},
proved respectively in Sections \ref{sec_trunc1} and \ref{Sec_couple}.
Moreover, the first part of \Thm~\ref{thm_elogz} follows from \Prop~\ref{prop_recon_energy} and a few lines of calculations;
this is show just below.

The above argument highlights the difference between the \ER\ graph and the random regular graph and the reason why we have the strict inequality $\b^\dagger(d)<\b^*(d)$ for all $d\geq3$.
Indeed, in the $d$-regular tree, to which the random regular graph converges locally, every vertex has $d-1$ children.
By contrast, the \ER\ graph of average degree $d$ converges locally to a Galton-Watson tree with offspring distribution $\Po(d)$.
Hence, the average number of children has mean $d$ rather than $d-1$.
The effect is that the broadcasting process on the Galton-Watson tree is able to remember the spin of the root for smaller values of $\b$ than in the regular case.
Therefore, it is natural to expect that on the \ER\ graph long-range correlations emerge for smaller $\b$.

\begin{proof}[Proof of \Thm~\ref{thm_elogz}, part 1]
\newcommand{\Zb}{Z_{\GG,\b}}
\newcommand{\lZb}{\log \Zb}
\newcommand{\elZb}{\ex[\lZb]}
\newcommand{\oneO}{\vecone\!\!\cbc{\cO}}
First, we argue by Azuma's inequality that 
for any $d\geq3$ and $\b>0$, $\lZb$ is concentrated about $\elZb$.
As is standard, construct $\GG$ using $dn/2$ independent random variables $X_i$ each giving the matching of the next point in the configuration model.
Compare this to a uniform random reference matching.
If $X_i$ matches a point $A$ to $B$ but the reference matching matched $A$ to $C\neq B$ and $B$ to $D$, update the reference by matching $A$ to $B$ and $C$ to $D$.
The reference copy has two edges added and two deleted, and with $X_i$ uniformly random the reference remains uniformly random, so $X_i$ changes the expectation of $\lZb$ conditioned on $X_1,\ldots,X_i$ by at most $2\b$.
Azuma's inequality yields

\begin{align}\label{eqthm_elogz_p_2}
	\pr\brk{\abs{\log Z_{\GG,\b}-\ex[\log Z_{\GG,\b}]}>t}&\leq2\exp\bc{-\frac{t^2}{4\b^2dn}}\qquad(t>0).
\end{align}

For $\b<\b^*(d)$, \Prop~\ref{prop_recon_energy} gives $\ex[\Zb]=\Theta(\ex[\Zb \oneO)]$ while trivially $\ex[\Zb^2]\leq \ex[\Zb^2 \oneO]$, so from the Paley-Zygmund inequality,
\begin{align}\label{eqthm_elogz_p_1}
	\pr\brk{Z_{\GG,\b}\geq \frac12 \ex[Z_{\GG,\b}]}
	&\geq\frac14\frac{\ex[\Zb]^2}{\ex[\Zb^2]}
	=\Omega(1)\frac{\ex[\Zb\oneO]^2}{\ex[\Zb^2 \oneO]}
	=\Omega(\exp(-o(n))),
\end{align}
the last inequality again from \Prop~\ref{prop_recon_energy}.
Thus, $\pr\brk{\log \Zb \geq \log\ex[\Zb]-2}=\Omega(\exp(-o(n)))$.
For any $\eps>0$, were there arbitrarily large $n$ for which $\log \ex[\Zb] > \ex[\log \Zb] +\eps n$ this would contradict the result from Azuma.
But by Jensen's inequality $\log \ex[\Zb] \geq \ex[\log \Zb]$,
so $\ex[\log \Zb] = \log \ex[\Zb]+o(n)$.
Taking the value of $\log \ex[\Zb]$ from \Lem~\ref{Cor_simple*} (or \Prop~\ref{prop_recon_energy}) establishes the first part of \Thm~\ref{thm_elogz}.

\ignore{
\begin{align*}
    \Pr \brk{\log Z_{\GG, \b} \geq \log \Erw\brk{Z_{\GG, \b}} - \log n} &\geq \Pr \brk{\log Z_{\GG, \b} \geq \log \Erw\brk{Z_{\GG, \b} \vecone \cbc{\cO}} + \log C - \log n} \\
    &\geq \Pr \brk{\log \bc{ Z_{\GG, \b} \vecone \cbc{\cO}} \geq \Erw\brk{\log \bc{ Z_{\GG, \b} \vecone \cbc{\cO}}} - \log n} \\
    &= \Pr \brk{\bc{ Z_{\GG, \b} \vecone \cbc{\cO}} \geq \Erw\brk{\bc{ Z_{\GG, \b} \vecone \cbc{\cO}}} \exp \bc{- \log n}} \\
    &\geq \bc{1-\exp(-\log n)}^2 \frac{\ex[Z_{\GG,\b}\vecone\cbc{\cO}]^2}{\ex[Z_{\GG,\b}^2\vecone\cbc{\cO}]} \\
    &\geq \frac{\ex[Z_{\GG,\b}\vecone\cbc{\cO}]^2}{2\ex[Z_{\GG,\b}^2\vecone\cbc{\cO}]} \\
    &\geq \exp(-\delta n)\qquad\mbox{with }\delta=o(1), C=\Theta(1)
\end{align*}

Combining \eqref{eqthm_elogz_p_1} and \eqref{eqthm_elogz_p_2}, we see that $\log Z_{\GG,\b}=\log \ex[Z_{\GG,\b}]+o(n)$ \whp\
Therefore, the first part of \Thm~\ref{thm_elogz} follows from \Lem~\ref{Cor_simple*}.
}
\end{proof}

\subsection{The Bethe free energy}

In the next step we prove the the second statement of \Thm~\ref{thm_elogz}.
As discussed earlier, 
the probability that a given graph $G$ comes up as the result $\GG^*$ of the stochastic block model is (nearly) proportional to the partition function $Z_{G,\b}$.
Therefore, if the partition function $Z_{\GG,\b}$ is tightly concentrated about its mean $\Erw[Z_{\GG,\b}]$, then we might expect that the distribution of $\GG^*$ and of the plain random $d$-regular graph are `close'.
By contrast, if 
$Z_{\GG,\b}$ is not concentrated but prone to a lottery phenomenon where a few unlikely outcomes render a disproportionate contribution to $\Erw[Z_{\GG,\b}]$, then we should expect that this discrepancy is exacerbated upon passing to size-biased model $\GG^*$ as outliers receive an extra boost.
The following lemma formalises this intuition.
It replaces the vague `concentration' phrasing with asymptotic equality of $\Erw\brk{\log Z_{\GG,\b}}$ and $\log \Erw \brk{Z_{\GG,\b}}$, 
and the equivalent for $\GG^*$, with $\log \Erw \brk{Z_{\GG,\b}}$
known from \eqref{eqTech8}; this equality certainly follows from sufficient concentration, while without concentration the equality would be an odd coincidence.

\begin{lemma}[{\cite[Lemma~4.4]{Coja_2020}}]\label{Lemma_GG*}
Let $d\geq3$ and $\b>0$.
We have $\Phi_d(\b)=\log2+\frac{d}{2}\log\frac{1+\eul^{-\beta}}{2}$ if and only if
\begin{align}\label{eqLemma_GG*1}
\lim_{n\to\infty}\frac{1}{n}\Erw\brk{\log Z_{\GG^*,\b}}=\log2+\frac{d}{2}\log\frac{1+\eul^{-\beta}}{2}.
\end{align}
\end{lemma}

By the first part of \Thm~\ref{thm_elogz} the lemma's hypothesis holds for $\b<\b^*(d)$.
To prove the second part of  \Thm~\ref{thm_elogz} we will 
show that the conclusion (and thus the hypothesis)
is violated if $\b>\b^*(d)$.
As a stepping stone we use a variational formula for $\Erw[\log Z_{\GG^*,\b}]$ from~\cite{Coja_2020}.
Let $\cP_*([-1,1])$ be the space of all probability measures on the interval $[-1,1]$ with mean zero.
Moreover, for a given such probability measure $\pi$ let $(\MU_{\pi,i})_{i\geq1}$ be a family of independent samples from $\pi$ and let $\Lambda(x)=x\log x$.
The expression
\begin{align} \label{eq_Bethe_exp}
    \cB_{\text{Ising}}(\pi, \beta, d) = \Erw \brk{\frac{\Lambda \bc{\sum_{\sigma \in \pmone} \prod_{i=1}^d 1 - (1-e^{-\beta})(1+\sigma \MU_{\pi,i})/2}}{2^{1-d}(1+\eul^{-\beta})^d} 
			- \frac{d \Lambda \bc{1-(1-\eul^{-\beta})(1+\MU_{\pi,1}\MU_{\pi,2})/2 }}{1+\eul^{-\beta}}}
\end{align}
is called the Bethe free energy.

\begin{lemma}[{\cite[\Thm~2.3]{Coja_2020}}]\label{lem_Bethe}
For any $\beta>0$ and any $d \geq 3$, we have
\begin{align*}
    \limn\frac{1}{n}\Erw[\log Z_{\GG^*,\b}]=\sup_{\pi \in \cP_*([-1,1])}\cB_{\text{Ising}}(\pi, \beta, d) .
\end{align*}
\end{lemma}

Combining \Lem s~\ref{Lemma_GG*} and~\ref{lem_Bethe}, we see that the second part of \Thm~\ref{thm_elogz} boils down to showing that 
\begin{align} \label{eq_Bethe}
    \sup_{\pi \in \cP_*(\pmone)}\cB_{\text{Ising}}(\pi, \beta, d) > \log 2+\frac{d}{2}\log\frac{1+\eul^{-\b}}2\qquad\mbox{for $\beta > \bstar(d)$}.
\end{align}
Luckily, the variational formula \eqref{eq_Bethe} asks to take a supremum over distributions $\pi$.
Therefore, it suffices to point to a specific distribution $\pi$ such that $\cB_{\text{Ising}}(\pi, \beta, d)$ exceeds the first moment bound.
Specifically, for a small $\eps>0$ let us introduce
\begin{align} \label{eq_distribution}
    \pi_\eps^* = \frac{1}{2} \bc{\delta_{ 2\eps} + \delta_{- 2\eps}}
\end{align}
where $\delta_{z}$ is the point mass on $z$.
It is easy to see that for $\eps=0$ we precisely obtain $\cB_{\text{Ising}}(\pi_0^*, \beta, d)=\log(2)+d\log((1+\eb)/2)/2$.
The following proposition shows that for $\b>\b^*(d)$ small $\eps>0$ yield a slightly but strictly larger value.

\begin{proposition} \label{prop_Taylor}
For any $\beta>\bstar(d)$ there exists $\eps>0$ such that 
\begin{align*}
    \cB_{\text{Ising}}(\pi_\eps^*, \beta, d) > \log 2+\frac{d}{2}\log\frac{1+\eul^{-\b}}2.
\end{align*}
\end{proposition}

\begin{proof}[Proof of \Thm~\ref{thm_elogz}, part 2]
This follows directly from \Lem~\ref{lem_Bethe} and \Prop~\ref{prop_Taylor}.
\end{proof}

\begin{proof}[Proof of \Thm~\ref{Thm_sbm}]
The theorem follows from \Thm~\ref{thm_elogz}, Theorem 17.1 in \cite{Coja_2020} and the fact that the disassortative stochastic block model with two communities is the planted Ising antiferromagnet.
\end{proof}

\subsection{The interpolation method}
With \Thm~\ref{thm_elogz} in place, let us consider the case that $\beta \to \infty$ which eventually allows us to derive improved upper bounds on the expected maximum cut size of random regular graphs.
Unfortunately, the stochastic dependencies between vertex spins make it difficult to get a handle on a simple expression like we had for $\beta < \bstar(d)$ where we simply obtained the first moment bound.
Apart from the obvious short-range dependencies that, for example, induce adjacent vertices to prefer opposite spins, we expect long-range dependencies to occur above the Kesten-Stigum bound.
Thus, for $\b>\b^*(d)$ the spin of a vertex impacts those of distant vertices.

The `1-step replica symmetry breaking ansatz' from physics attempts to describe these long-range dependencies by means of an additional hidden variable~\cite{pnas,MP1}.
The basic hypothesis is that for $\b>\b^*(d)$ the phase space, i.e., the set $\{\pm1\}^{V_n}$ of all possible spin configurations, decomposes into a number $\vec S_1,\ldots,\vec S_\ell$ of `pure states' \whp\
Mathematically $\vec S_1,\ldots,\vec S_\ell$ are pairwise disjoint subsets of $\{\pm1\}^{V_n}$ such that $\mu_{\GG,\b}(\vec S_1)+\cdots+\mu_{\GG.\b}(\vec S_\ell)\sim 1$.
Thus, the sets cover nearly the entire support of the Boltzmann distribution.
Furthermore, once we condition on a pure state $\vec S_h$, long-range effects disappear; formally,
\begin{align*}
	\ex\abs{\sum_{h=1}^\ell\mu_{\GG,\b}(\vec S_h)\bc{\pr\brk{\vec\sigma_{\GG,v_1}=s,\vec\sigma_{\GG,v_2}=s'\mid\GG,\vec\sigma_{\GG}\in\vec S_h}-\pr\brk{\vec\sigma_{\GG,v_1}=s\mid\GG,\vec\sigma_{\GG}\in\vec S_h}\pr\brk{\vec\sigma_{\GG,v_2}=s'\mid\GG,\vec\sigma_{\GG}\in\vec S_h}}}&=o(1).
\end{align*}
Hence, long-range correlations of the unconditional measure $\mu_{\GG,\b}$ arise because the spin $\vec\sigma_{\GG,v_1}$ of a vertex hints at the pure state $\vec S_h$ to which $\vec\sigma_{\GG}$ belongs, which in turn skews our expectations as to the other spins. 
The existence of such a pure state decomposition has been established rigorously~\cite{Coja_2019}.

How does this picture help to estimate the partition function $Z_{\GG,\b}$?
The basic idea behind the interpolation method is to set up an synthetic model of a spin system that exhibits precisely the long-range dependencies predicted by the 1-step rsb ansatz, and no others.
Mathematically this model is represented by a factor graph (or a Markov random field); see the left panel of Figure~\ref{Fig_interpolation}.
The factor graph contains variable nodes (the white circles) that represent the vertices of our graph.
Each of these variable nodes is connected to an external field (the blue box) that is meant to represent the impact of the short-range dependencies imposed by one of the incident edges of the corresponding vertex of $\GG$.
But instead of the complicated direct interactions between the vertices through actual edges as in the original random graph $\GG$, the variable nodes only interact with each other through the yellow node.
This node represents the hidden variable postulated by the 1-step rsb ansatz, i.e., the index of the pure state.
Finally, the red boxes are `negative edges'.
They are necessary because the variable nodes do not interact directly.
In effect, the number of blue nodes is twice the number of edges of the actual graph $\GG$, and thus we have to compensate for the impact of $dn/2$ spare blue boxes.

The cunning idea behind the interpolation method is to build a family of factor graph models parametrised by time $t\in[0,1]$.
The interpolation scheme starts from the artificial factor graph model at time $t=0$.
At each intermediate time step $t\in(0,1)$ the model blends the synthetic $t=0$ case and the actual Ising model on $\GG$.
Ultimately at time $t=1$ all the synthetic ingredients (the blue and red boxes) disappeared and we are left with just the Ising antiferromagnet on $\GG$.
Remarkably, it is possible to prove that the partition function decreases monotonically in terms of $t$.
As a consequence, the partition function of the synthetic model upper bounds that of the Ising model on $\GG$.
Fortunately we do not need to carry out the interpolation method in full.
The result that we need follows from a more general version of the interpolation bound derived in~\cite{Sly_2016}.

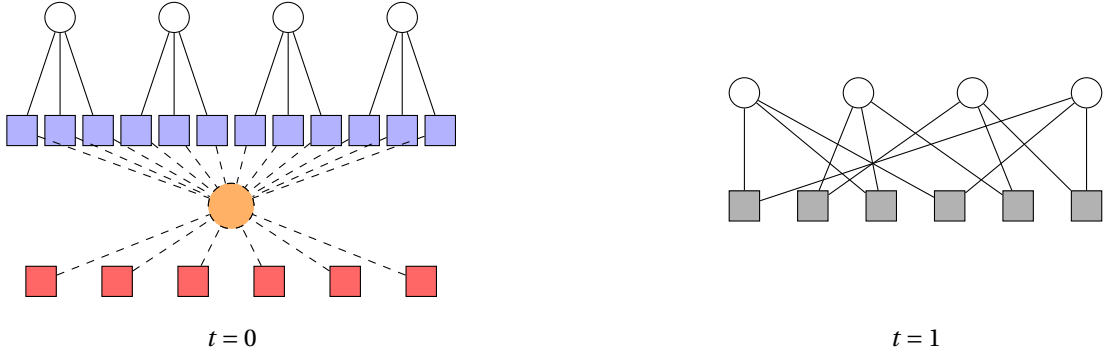
\begin{figure}
\begin{minipage}{0.99 \linewidth}
\begin{tikzpicture}[]
    \node[draw=black, shape=circle, fill=none, minimum width=0.4cm, minimum height = 0.4cm] (VV1) at (10, -0.5) {};
    \node[draw=black, shape=circle, fill=none, minimum width=0.4cm, minimum height = 0.4cm] (VV2) at (11.5, -0.5) {};
    \node[draw=black, shape=circle, fill=none, minimum width=0.4cm, minimum height = 0.4cm] (VV3) at (13, -0.5) {};
    \node[draw=black, shape=circle, fill=none, minimum width=0.4cm, minimum height = 0.4cm] (VV4) at (14.5, -0.5) {};
    
    \node[draw=black, fill=black!30, shape=rectangle, minimum width=0.4cm, minimum height = 0.4cm] (FF1) at (10, -2) {};
    \node[draw=black, fill=black!30, shape=rectangle, minimum width=0.4cm, minimum height = 0.4cm] (FF2) at (10.9, -2) {};
    \node[draw=black, fill=black!30, shape=rectangle, minimum width=0.4cm, minimum height = 0.4cm] (FF3) at (11.8, -2) {};
    \node[draw=black, fill=black!30, shape=rectangle, minimum width=0.4cm, minimum height = 0.4cm] (FF4) at (12.7, -2) {};
    \node[draw=black, fill=black!30, shape=rectangle, minimum width=0.4cm, minimum height = 0.4cm] (FF5) at (13.6, -2) {};
    \node[draw=black, fill=black!30, shape=rectangle, minimum width=0.4cm, minimum height = 0.4cm] (FF6) at (14.5, -2) {};
    
    \node[draw = none] at (12.25, -3.75)  {$t = 1$};
    
    \draw[-] (VV1)--(FF1);
    \draw[-] (VV1)--(FF3);
    \draw[-] (VV1)--(FF4);
    \draw[-] (VV2)--(FF5);
    \draw[-] (VV2)--(FF3);
    \draw[-] (VV2)--(FF2);
    \draw[-] (VV3)--(FF6);
    \draw[-] (VV3)--(FF2);
    \draw[-] (VV3)--(FF5);
    \draw[-] (VV4)--(FF1);
    \draw[-] (VV4)--(FF4);
    \draw[-] (VV4)--(FF6);
    
    \node[draw=black, fill=red!60, shape=rectangle, minimum width=0.4cm, minimum height = 0.4cm] (A1) at (0.75, -3) {};
    \node[draw=black, fill=red!60, shape=rectangle, minimum width=0.4cm, minimum height = 0.4cm] (A2) at (1.75, -3) {};
    \node[draw=black, fill=red!60, shape=rectangle, minimum width=0.4cm, minimum height = 0.4cm] (A3) at (2.75, -3) {};
    \node[draw=black, fill=red!60, shape=rectangle, minimum width=0.4cm, minimum height = 0.4cm] (A4) at (3.75, -3) {};
    \node[draw=black, fill=red!60, shape=rectangle, minimum width=0.4cm, minimum height = 0.4cm] (A5) at (4.75, -3) {};
    \node[draw=black, fill=red!60, shape=rectangle, minimum width=0.4cm, minimum height = 0.4cm] (A6) at (5.75, -3) {};
    
    \node[draw=black, dashed, shape=circle, fill=orange!60, minimum width=0.6cm, minimum height = 0.6cm] (S2) at (3.25, -2) {};
    
    \draw[dashed] (A1)--(S2);
    \draw[dashed] (A2)--(S2);
    \draw[dashed] (A3)--(S2);
    \draw[dashed] (A4)--(S2);
    \draw[dashed] (A5)--(S2);
    \draw[dashed] (A6)--(S2);

    \node[draw=black, shape=circle, fill=none, minimum width=0.3cm, minimum height = 0.4cm] (V1) at (1, 0.5) {};
    \node[draw=black, shape=circle, fill=none, minimum width=0.3cm, minimum height = 0.4cm] (V2) at (2.5, 0.5) {};
    \node[draw=black, shape=circle, fill=none, minimum width=0.3cm, minimum height = 0.4cm] (V3) at (4, 0.5) {};
    \node[draw=black, shape=circle, fill=none, minimum width=0.3cm, minimum height = 0.4cm] (V4) at (5.5, 0.5) {};
    
    \node[draw=black, fill=blue!30, shape=rectangle, minimum width=0.4cm, minimum height = 0.4cm] (F1) at (0.5, -1) {};
    \node[draw=black, fill=blue!30, shape=rectangle, minimum width=0.4cm, minimum height = 0.4cm] (F2) at (1, -1) {};
    \node[draw=black, fill=blue!30, shape=rectangle, minimum width=0.4cm, minimum height = 0.4cm] (F3) at (1.5, -1) {};
    \node[draw=black, fill=blue!30, shape=rectangle, minimum width=0.4cm, minimum height = 0.4cm] (F4) at (2, -1) {};
    \node[draw=black, fill=blue!30, shape=rectangle, minimum width=0.4cm, minimum height = 0.4cm] (F5) at (2.5, -1) {};
    \node[draw=black, fill=blue!30, shape=rectangle, minimum width=0.4cm, minimum height = 0.4cm] (F6) at (3, -1) {};
    \node[draw=black, fill=blue!30, shape=rectangle, minimum width=0.4cm, minimum height = 0.4cm] (F7) at (3.5, -1) {};
    \node[draw=black, fill=blue!30, shape=rectangle, minimum width=0.4cm, minimum height = 0.4cm] (F8) at (4, -1) {};
    \node[draw=black, fill=blue!30, shape=rectangle, minimum width=0.4cm, minimum height = 0.4cm] (F9) at (4.5, -1) {};
    \node[draw=black, fill=blue!30, shape=rectangle, minimum width=0.4cm, minimum height = 0.4cm] (F10) at (5, -1) {};
    \node[draw=black, fill=blue!30, shape=rectangle, minimum width=0.4cm, minimum height = 0.4cm] (F11) at (5.5, -1) {};
    \node[draw=black, fill=blue!30, shape=rectangle, minimum width=0.4cm, minimum height = 0.4cm] (F12) at (6, -1) {};
    
    \node[draw=black, dashed, shape=circle, fill=orange!60, minimum width=0.6cm, minimum height = 0.6cm] (S1) at (3.25, -2) {};
    
    \node[draw = none] at (3.25, -3.75)  {$t = 0$};
    
    \draw[-] (V1) -- (F1);
    \draw[-] (V1) -- (F2);
    \draw[-] (V1) -- (F3);
    \draw[-] (V2) -- (F4);
    \draw[-] (V2) -- (F5);
    \draw[-] (V2) -- (F6);
    \draw[-] (V3) -- (F7);
    \draw[-] (V3) -- (F8);
    \draw[-] (V3) -- (F9);
    \draw[-] (V4) -- (F10);
    \draw[-] (V4) -- (F11);
    \draw[-] (V4) -- (F12);
    
    \draw[dashed] (F1) -- (S1);
    \draw[dashed] (F2) -- (S1);
    \draw[dashed] (F3) -- (S1);
    \draw[dashed] (F4) -- (S1);
    \draw[dashed] (F5) -- (S1);
    \draw[dashed] (F6) -- (S1);
    \draw[dashed] (F7) -- (S1);
    \draw[dashed] (F8) -- (S1);
    \draw[dashed] (F9) -- (S1);
    \draw[dashed] (F10) -- (S1);
    \draw[dashed] (F11) -- (S1);
    \draw[dashed] (F12) -- (S1);
    \end{tikzpicture}
    
\end{minipage}
    \caption{The factor graphs $\GG_0$ and the original factor graph $\GG_1$ with $n=4, d=3$.
}
    \label{Fig_interpolation}
\end{figure}

To state the resulting upper bound precisely, fix any probability measure $\fr$ on $[-1,1]$.
Let $(\vec r_i)_{i\geq[d]}$ be a family of independent random variables with distribution $\fr$; thus, $\vec r_i\in[-1,1]$ for all $i$.
Further, define
\begin{align*}
	\RHO_i(\sigma)&=\frac{1+\sigma\vec r_i}{2}\qquad(\sigma=\pm1).
\end{align*}
The idea is that $\vec \rho_1,\ldots,\vec \rho_d$ represent the short-range influences that the neighbours of some vertex, say $v_1$, exercise on the spin of that vertex within a single pure state.
More specifically, think of $\vec\rho_i(s)$ as the probability that the $i$-th neighbour of $v_1$ would take spin $s\in\{\pm1\}$ if we removed $v_1$ from the random graph.
The following lemma is an immediate consequence of \cite[Theorem E.5]{Sly_2016}.

\begin{lemma}\label{Prop_PD}
	Let $d\geq3,\b>0$.
	Then for any $y>0$ and any $\fr\in\cP([-1,1])$ we have $\Phi_d(\b)\leq\phi_{\b,y}(\fr)$, where
	\begin{align}\label{eqProp_PD}
		\phi_{\b,y}(\fr)&=\frac{1}{y}\log \Erw[X_1^y] -\frac{d}{2y}\log \Erw[X_2^{y}],\\
		X_1&=\sum_{\tau\in\{\pm 1\}} \prod_{h=1}^{d}1-(1-\eul^{-\b})\RHO_{h}(\tau),\qquad X_2=1-(1-\eul^{-\b})\sum_{\tau\in\{\pm 1\}} \RHO_{1}(\tau)\RHO_{2}(\tau)\nonumber.
	\end{align}
\end{lemma}

Clearly, \eqref{eqProp_PD} is not exactly what we had in mind when aiming for an explicit expression of the upper bound for $\Phi_d(\beta)$. 
However, a key feature of \Lem~\ref{Prop_PD} is that the inequality holds for \textit{any} $y,\fr$.
We are thus free to choose these parameters so that we obtain a reasonable expression and, hopefully, at the same time a good upper bound.

Following physics intuition~\cite{MP1,MP2} we define the measure $\fr$ as follows.
Let $\delta_x\in\cP([-1,1])$ be the atom on $x\in[-1,1]$.
Then for $\a\in[0,1/2]$ we let
\begin{align}\label{eqrhoalpha}
	\fr_\a&=\a \delta_{-1}+(1-2\a) \delta_{0}+\a \delta_1\in\cP([-1,1]).
\end{align}
Intuitively, we `freeze' a spin to $+1$ or $-1$ with probability $\a$.
Otherwise, if the spin does not freeze we leave it unbiased, i.e., it takes either spin $\pm1$ with equal probability.
The following proposition shows that for the distribution $\fr$ from \eqref{eqrhoalpha} the function $\phi_{\b,y}(\fr)$ boils down to a manageable expression.

\begin{proof}[Proof of \Thm~\ref{Thm_1rsb}]
The theorem is an immediate consequence of \Lem~\ref{Prop_PD} and \Prop~\ref{Prop_Explicit_Yprime}. 
\end{proof}

What does the bound look like in the trivial case $\a=0$?
For any $y>0$ we obtain
\begin{align}
	\Phi_d(\b)
	\leq \phi_{\b,y}(\fr_0)
	&\leq \frac{1}{y}\bc{\log \Erw[X_1^y] - \frac{d}{2}\log \Erw[X_2^{y}] }= \frac{1}{y} \log \bc{2 \bc{\frac{1+\eb}{2}}^d}^y - \frac{d}{2y} \log \bc{\frac{1+\eb}{2}}^y \\
	&= \log 2 + \frac{d}{2} \log \frac{1+\eb}{2} .\label{eq_elog_loge_int}
\end{align}
Hence, we simply recover the first moment bound~\eqref{eqJensen}.
However, for large $\b$ the strictly positive $\a$ render a better bound.
The following proposition simplifies the expression from  \Thm~\ref{Thm_1rsb} for large $\b$. 
Recall $F$ from \eqref{eq_F_def}.

\begin{proposition}\label{Prop_Explicit_Yprime}
Let $d\geq 3,\b>0,0<z<1,0<\a<1/2$.
Then with $y=y(\b)=-\log(z)/\b$ we have
\[
\limb \frac1{\b y}\bc{\log \Erw \brk{X_1^{y}} -\frac{d}{2}\log \Erw[X_2^{y}]}
= F_d(\alpha, z)
. \]
\end{proposition}

\begin{proof}[Proof of Corollary~\ref{cor_max_cut}]
For any $d$-regular graph $G$ on $n$ vertices and any $\b>0$, we have
\begin{align*}
    \frac2{dn} \mbox{\sc MaxCut}(G)
    =1-\frac2{dn}\min_{\sigma\in\cbc{\pm1}^n}\cH_{G}(\sigma)
    \leq 1 + \frac {2}{\b dn}\log Z_\b(G).
\end{align*}

Thus by \Thm~\ref{Thm_1rsb}, we obtain
\begin{align}\label{eqcor_max_cut1}
    \limsup_{n\to\infty}\frac2{dn}\Erw[\ecutm]
    &\leq 1 + \frac {2}{d}\lim_{\b\to\infty}\Phi_d(\b)/\b \\
    &\leq 1 + \frac 2d  \inf_{\substack{0<\a\leq1/2\\0<z<1}} F_d(\a,z) .
\end{align}
The corollary follows.
\end{proof}

\subsection{Organisation}

Let us outline how the remainder of the paper is organized. \Sec~\ref{sec_thm_1} is devoted to proving \Prop~\ref{prop_recon_energy}, while \Sec~\ref{sec_thm_2} contains the proof of \Prop~\ref{prop_Taylor}. Jointly, the two sections provide the missing pieces for the proof of \Thm~\ref{thm_elogz}. Finally, in \Sec~\ref{sec_cor} we prove the two key \Lem s~\ref{Lem_Explicit_Yprime} and \ref{Lem_Explicit_Y} needed for the proof of \Cor~\ref{cor_max_cut}.

\subsection{Notation}
We will denote a random $d$-regular graph on $n$ vertices by $\GG(n,d)$.
When the context is clear, we will simply write $\GG = \GG(n,d)$.
We tacitly assume that $dn$ is even.
Throughout the paper we will use standard Landau notation with the usual symbols $o(\cdot), O(\cdot), \Theta(\cdot), \omega(\cdot)$ and $\Omega(\cdot)$.
These symbols refer to the limit $n\to\infty$ by default, but may refer to other limits where specified.

For a subset $I\subset\RR$ we denote by $\cP(I)$ the set of all Borel probability measures on $I$.
Moreover, for a finite set $\Omega\neq\emptyset$ let $\cP(\Omega)$ be the set of all probability distributions on $\Omega$.
We recall that the entropy $H(\mu)$ of such a probability distribution $\mu \in \cP(\Omega)$ is defined as
\begin{align*}
    H(\mu) = - \sum_{\omega \in \Omega} \mu(\omega) \log \mu(\omega).
\end{align*}
We will also need the Kullback-Leibler divergence of $\mu, \nu \in \cP(\Omega)$, defined as
\begin{align*}
	\KL{\mu}{\nu} = \sum_{\omega \in \Omega} \mu(\omega) \log \frac{\mu(\omega)}{\nu(\omega)}\in[0,\infty],
\end{align*}
with the conventions $0\log 0=0$, $0\log\frac{0}{0}=0$ and $-\log0=\infty$.

\section{Discussion} \label{sec_discussion}

\noindent
In this section we relate the contributions of the present paper to prior work.
We begin with the statistical physics perspective.

\subsection{Replica symmetry breaking}
The Ising model, proposed by Lenz in 1920 \cite{Lenz_1920}, has become a cornerstone of statistical physics generally~\cite{Friedli_2017,Huang_2009}.
Moreover, the Ising model on random graphs in particular has proved a testbed for the investigation of the idea of replica symmetry breaking that was proposed by \Mezard\ and Parisi on the basis of the non-rigorous `cavity method'~\cite{MP1,MP2}.
The corroboration of the cavity method's predictions for the ferromagnetic Ising model on \ER\ graphs by Dembo and Montanari~\cite{Dembo_2010} was a first success, although replica symmetry breaking does not occur in this model.
The proof was based on the analysis of the Belief Propagation recurrences on random trees.
These techniques have subsequently been extended to the Potts model, a generalisation of the Ising model with more than two possible spin values~\cite{Dembo_2013, Dembo_2014}.

The antiferromagnetic version of the Potts and Ising models is closely related to the stochastic block model.
The results of Mossel, Neeman and Sly~\cite{Mossel_2016} on the block model with two communities therefore imply the existence and location of a replica symmetry breaking phase transition in the Ising antiferromagnet on the \ER\ graph.
Thus, as we saw above a new contribution of the present paper is the extension to random regular graphs.
Moreover, results from Coja-Oghlan, Krzakala, Perkins and Zdeborov\'a~\cite{Coja_2018_cavity} imply the existence and location of a replica symmetry breaking phase transition for the Potts model on \ER\ graphs.
The recent work of Coja-Oghlan, Hahn-Klimroth, Loick, M\"uller, Panagiotou and Pasch~\cite{Coja_2020} extend these results to graphs with given degree sequences.
However, the results from~\cite{Coja_2020} determine the location of the replica symmetry breaking phase transition only implicitly as the solution to an infinite-dimensional variational problem.
Thus, the contribution of \Thm~\ref{thm_elogz} is the explicit analytic formula for the phase transition $\beta^*(d)$, which matches the combinatorially meaningful Kesten-Stigum bound~\cite{Kesten_1966}.

Apart from the Potts and Ising models, replica symmetry breaking phase transitions have been pinpointed in several other models.
Examples include random (hyper)graph colouring, several other random constraint satisfaction problems and further models from mathematical physics, such as the Viana-Bray spin glass model \cite{Guerra_2004}.
But usually the formula for the phase transition comes in as a complicated variational problem.
Indeed, the question whether the replica symmetry breaking transition equals the explicit Kesten-Stigum threshold has been linked to the order of the phase transition~\cite{Ricci_2019}, a question that merits further rigorous attention.

\subsection{The {\sc Max Cut} problem}
The semidefinite programming based {\sc Max Cut} algorithm of Goemans and Williamson \cite{Goemans_1995} has been one of the most important contributions to algorithms research.
The algorithm achieves an approximation ratio of $\min_{0 \leq \theta \leq \pi} \frac 2\pi \frac{\theta}{1-\cos(\theta)}\approx0.878$ on graphs with non-negative edge weights.
On regular graphs better approximation ratios can be achieved (also via semidefinite programming)~\cite{Feige_2002}.
The question whether the Goemans-Williamson approximation ratio is optimal has sparked an important line of research.
H\r{a}stad~\cite{Hastad_2001} derived from the PCP theorem that no approximation better than $0.941$ can be attained unless P$=$NP.
Moreover, Koth~\cite{Koth_2002} showed that the unique games conjecture implies the optimality of the Goemans-Williamson approximation ratio; see Barak \cite{Barak_2014} for a discussion.

Given the great interest in {\sc Max Cut} generally, it is hardly surprising that the problem has been studied intensively on random graphs, too.
In the classical combinatorics literature upper bounds have typically been based on the first moment method, while greedy algorithms were employed to derive lower bounds \cite{Bertoni_1997,Coja_2003, Coppersmith_2004,Csoka_2016,Diaz_2003,Diaz_2007,Kalapala_2002,Kardos_2012}.
A semidefinite programming approach was taken in \cite{Montanari_2016} on \Erdos-\Renyi\ graphs.
Table~\ref{Tab_mc} summarises the best explicit prior bounds for random regular graphs.
Naturally, arguments based on the method of moments or greedy algorithms suffer from the shortcoming of being inherently local, i.e., confined to short-range interactions. 
In effect, they remain oblivious to the long-range interactions that, according to physics prediction, shape the {\sc Max Cut} problem on random graphs.
Therefore, it is unsurprising that these techniques only carry so far.

The first complex model where the long-range interactions predicted by the theory of replica symmetry breaking were well understood is the Sherrington-Kirkpatrick spin glass.
The model can be viewed as a weighted {\sc Max Cut} problem on a complete graph.
Specifically, the weight of the edge between vertices $v,w$ is a Gaussian $\vJ_{v,w}$.
The random variables $(\vJ_{v,w})_{1\leq v<w\leq n}$ are mutually independent.
Hence, the model is described by the random Hamiltonian
\begin{align*}
    \cH_{\text{SK}}(\sigma) &= - \frac{1}{\sqrt{n}} \sum_{1\leq i<j \leq n} \vJ_{ij} \sigma_i \sigma_j \qquad(\sigma \in \cbc{\pm 1}^{n}),
\end{align*}
which induces a partition function and a Boltzmann distribution as in \eqref{eq_boltzmann}.
Clearly, the `ground state energy' $\min_{\sigma}\cH_{\text{SK}}(\sigma)$ corresponds to the maximum cut weight.
Parisi's seminal work~\cite{Parisi_1980} predicted formulas for the free energy and the ground state energy of the Sherrington-Kirkpatrick model.
After several decades Talagrand established the `Parisi formula' rigorously~\cite{Talagrand_2006}.
An important ingredient to this work was the interpolation method, which Guerra had proposed~\cite{Guerra_2003}.
Panchenko developed a different argument~\cite{Panchenko_2013_2}, which also led to a proof of Parisi's ultrametricity conjecture~\cite{Panchenko_2013}.

Franz and Leone~\cite{Franz_2003} extended the interpolation method to sparse random graphs; see also~\cite{Panchenko_2004}.
The version of the interpolation method quoted in \Lem~\ref{Prop_PD} is an adaptation to random regular graphs. 
Furthermore, Dembo, Montanari and Sen~\cite{Dembo_2017} used interpolation techniques to strike a chord between the sparse \ER\ graph and the Sherrington-Kirkpatrick model.
Specifically, they proved that
\begin{align}\label{eqDembo}
    \lim_{d\to\infty}\lim_{n \to \infty} \frac2{\sqrt dn}\brk{ \mbox{\sc MaxCut}(\GG(n,d/n))-\frac{dn}{4} }= p_\star\approx 0.7632,
\end{align}
where $p_\star$ derives from the ground state energy of the Sherrington-Kirkpatrick model.
Conceptually it seems natural to expect that the Sherrington-Kirkpatrick model occurs as the limit of sparse random graphs as the average degree gets large, basically due to central limit theorem-like effects.
Yet this result says nothing about any finite $d$ and, indeed, sparse random graphs for fixed finite values of $d$ appear to exhibit a more diverse and potentially even more intricate behavior.
As a result, we are only just beginning to understand the genuine behaviour of sparse models in the replica symmetry breaking phase; see, e.g., \cite{Bartha_2019}.

Finally, Panchenko~\cite{Panchenko_2013_3} obtained a variational formula for the free energy of the Ising antiferromagnet on \ER\ graphs.
The formula involves an optimisation over exchangeable distributions on $\{\pm1\}^{\NN\times\NN}$ subject to certain invariance conditions.
Coja-Oghlan and Perkins~\cite{Coja_2019} extended this result to random regular graphs, also pointing out that a corresponding variational formula can be derived for the {\sc Max Cut} of $\GG(n,d)$ for any fixed $d$.
However, the formula is not explicit, and it appears difficult (to put it mildly) to extract any numerical estimates.
Thus, the contribution of \Cor~\ref{cor_max_cut} is that we obtain a (relatively) simple explicit formula that incorporates at least the first level of the physicists' replica symmetry breaking formalism.

\subsection{The stochastic block model}
The Ising antiferromagnet is intimately related to the stochastic block model which has gained significant attention in recent years~\cite{Abbe_2017}.
The model provides a benchmark for both Bayesian inference and graph clustering, the basic idea being to create a random graph with a community structure.
In the simplest version the vertex set is partitioned into $q$ communities and edges between vertices in the same community are either more likely (assortative) or less (disassortative).
The question is for what discrepancy of edge densities it is possible to at least partially recover the community structure or, less ambitiously, to at least discriminate a random graph drawn from the block model from a null model.
The modern study of the stochastic block model originated with conjectures that Decelle, Krzakala, Moore and Zdeborov\'a~\cite{Decelle_2011} derived via the cavity method.
Specifically, they predicted a phase diagram that splits the model parameters into regions where recovering the community structure is information-theoretically and/or algorithmically feasible.

Mathematically the most complete picture exists for graphs with independent edges in the case of $q=2$ communities.
In this case the information-theoretic and algorithmic thresholds were established in a series of papers by Mossel, Neeman and Sly~\cite{Mossel_2015, Mossel_2015_2, Mossel_2016, Mossel_2018} and Massoulie~\cite{Massoulie_2014}.
For $q>2$ communities algorithms that match the conjecture from~\cite{Decelle_2011} have been proposed by Abbe and Sandon~\cite{Abbe_2018} and Bordenave, Lelarge and Massoulie~\cite{Bordenave_2015}.
As explained above, the contribution of \Thm~\ref{Thm_sbm} is to show that for the disassortative regular case with two communities the information-theoretic threshold equals the explicit Kesten-Stigum bound $\b^*(d)$.
Finally, as an interesting direction for future research we point to the question of developing an efficient algorithm that (partially) recovers the community structure $\vec\sigma^*$ for $\b>\b^*(d)$.

\section{Proof of \Prop~\ref{prop_recon_energy}} \label{sec_thm_1}

\noindent
The proof of \Prop~\ref{prop_recon_energy} requires several steps.
First we perform some preparatory calculations; in particular, we compute the first moment of the partition function of a random multi-graph drawn from the pairing model.
Subsequently we establish a relationship between the stochastic block model $\GG^*$ and the `null model' $\GG$.
Then we construct a coupling of the spin configuration around a typical vertex of $\GG^*$ with the broadcasting process from \Lem~\ref{lem_reconstruction} to estimate the probability that $\GG^*\in\cO$.
Finally, we perform a truncated moment computation to obtain \eqref{eqO}.

\subsection{The pairing model}\label{sec_firstmmt}
In order to calculate the first moment, as well as for some of the manoeuvres to follow, it will be convenient to replace the simple random $d$-regular graph $\GG$ by a random graph chosen from the pairing model.
Hence, think of the elements of $V_n\times[d]$ as vertex clones.
Moreover, let $\vec\Gamma$ be a random perfect matching of the complete graph on $V_n\times[d]$.
Finally, let $\G$ be the $d$-regular multigraph on $V_n$ obtained by contracting the clones $V_n\times[d]$.
With $\cS$ the set of all simple graphs, it is well known that
\begin{align}\label{eqsimple}
	\pr\brk{\G\in\cS}&=\Omega(1)&\mbox{and}&&\pr\brk{\GG\in\cE}&=\pr\brk{\G\in\cE\mid\cS}\qquad\mbox{for any event }\cE.
\end{align}
In order to compute the first moment $\ex[Z_{\GG,\b}]$ we will compute $\ex[Z_{\G,\b}]$ and then investigate the impact of conditioning on $\cS$.

To calculate $\ex[Z_{\G,\b}]$ we proceed as follows.
For $\sigma\in\{\pm1\}^{V_n}$ let $\rho(\sigma)=(\rho_{1}(\sigma),\rho_{-1}(\sigma))$ be the distribution on $\pm1$ defined by
\begin{align*}
	\rho_1(\sigma)&=\frac{1}{n}\sum_{i=1}^n\vecone\{\sigma_{v_i}=1\},&
	\rho_{-1}(\sigma)&=\frac{1}{n}\sum_{i=1}^n\vecone\{\sigma_{v_i}=-1\}.
\end{align*}
Thanks to the linearity of expectation we can write the first moment as
\begin{align*}
	\ex[Z_{\G,\b}]&=\sum_{\sigma\in\{\pm1\}^{V_n}}\ex[\psi_{\G,\b}(\sigma)].
\end{align*}
Naturally, $\psi_{\G,\b}(\sigma)$ depends on the number of edges that join vertices with the same spin.
Hence, to calculate $\ex[\psi_{\G,\b}(\sigma)]$ we need to know the number of graphs with a given number of such edges.

The following lemma solves this problem.
Let $\cM(\sigma)$ be the set of all probability distributions 
\begin{align}\label{eqmuconstr}
	\mu_{11}+\mu_{1-1}&=\rho_1,&
	\mu_{-1-1}+\mu_{1-1}&=\rho_{-1},&
	\mu_{1-1}&=\mu_{-11}
\end{align}
and such that $\mu_{11}dn,\mu_{-1-1}dn$ are even integers and $\mu_{1-1}dn$ is an integer.
Moreover, let $\cG(\sigma,\mu)$ be the event that $\G$ has $\mu_{11}dn/2$ edges that join vertices $v,w$ with $\sigma_v,\sigma_w=1$.
Then due to regularity there are $\mu_{-1-1}dn/2$ edges joining vertices that both carry a $-1$ spin and $\mu_{1-1}dn$ edges that connect vertices with opposite spins.

\begin{lemma}\label{Lemma_mu}
	For $\sigma\in\{\pm1\}^{V_n}$ and $\mu\in\cM(\sigma)$ we have 
	\begin{align*}
		\pr\brk{\G\in\cG(\sigma,\mu)}&=\binom{dn\rho_1(\sigma)}{dn\mu_{11}}\binom{dn\rho_{-1}(\sigma)}{dn\mu_{-1-1}}\frac{(dn\mu_{11}-1)!!(dn\mu_{-1-1}-1)!!(dn\mu_{1-1})!}{(dn-1)!!}.
	\end{align*}
\end{lemma}
\begin{proof}
	The denominator $(dn-1)!!$ simply counts the total number of possible perfect matchings $\vec\Gamma$.
	Moreover, the two binomial coefficients account for the number of ways of selecting clones of vertices with spin $\pm1$ to constitute edges of the four possible types.
	Finally, the numerator equals the number of possible ways to match the clones up according to these designated types.
\end{proof}

As it stands the formula from \Lem~\ref{Lemma_mu} does not yet lend itself to asymptotical calculations.
But Stirling's formula yields the following approximation.
\begin{corollary}\label{Cor_mu}
	For $\sigma\in\{\pm1\}^{V_n}$ and $\mu\in\cM(\sigma)$ we have 
	$ \pr\brk{\G\in\cG(\sigma,\mu)}=\exp\bc{-\frac{dn}{2}\KL{\mu}{\rho\tensor\rho}+O(\log n)}.  $
\end{corollary}

\noindent
\Cor~\ref{Cor_mu} follows from a more general lemma about partitions of random regular graphs from~\cite{Coja_2016}.
But since we will encounter similar calculations again in due course and because the proof is quite short, we include it here.
We need Stirling's formula
\begin{align}\label{eqStirling}
	k!&=\sqrt{2\pi k}\bcfr k\eul^k\exp(O(1/k))
\end{align}
and the elementary formula
\begin{align}\label{eqdouble}
	(2k-1)!!&=\frac{(2k)!}{k!2^k}.
\end{align}

\begin{proof}[Proof of \Cor~\ref{Cor_mu}]
	Applying \eqref{eqdouble}, we obtain
\begin{align*}
	(dn\mu_{11}-1)!!&=\frac{(dn\mu_{11})!}{2^{dn\mu_{11}/2}(dn\mu_{11}/2)!}&
	(dn\mu_{-1-1}-1)!!&=\frac{(dn\mu_{-1-1})!}{2^{dn\mu_{-1-1}/2}(dn\mu_{-1-1}/2)!},&
	(dn-1)!!&=\frac{(dn)!}{2^{dn/2}(dn/2)!}.
\end{align*}
Hence,
	\begin{align}\nonumber
		\frac{(dn\mu_{11}-1)!!(dn\mu_{-1-1}-1)!!(dn\mu_{1-1})!}{(dn-1)!!}&=2^{dn(1-\mu_{11}-\mu_{-1-1})/2}\binom{dn\mu_{1-1}}{dn\mu_{1-1}/2}^{-1}\binom{dn/2}{dn\mu/2}\binom{dn}{dn\mu}^{-1}\\
																		 &=2^{dn\mu_{1-1}}\binom{dn\mu_{1-1}}{dn\mu_{1-1}/2}^{-1}\binom{dn/2}{dn\mu/2}\binom{dn}{dn\mu}^{-1}.
																		 \label{eqCor_mu1}
	\end{align}
Thus, Stirling's formula \eqref{eqStirling} gives
\begin{align}
	\frac{(dn\mu_{11}-1)!!(dn\mu_{-1-1}-1)!!(dn\mu_{1-1})!}{(dn-1)!!}&=\exp\bc{-dnH(\mu)/2+O(\log n)}.
																		 \label{eqCor_mu2}
	\end{align}
	Further, combining \eqref{eqmuconstr} and \eqref{eqStirling}, we obtain
	\begin{align}\label{eqCor_mu3}
		\binom{dn\rho_1(\sigma)}{dn\mu_{11}}\binom{dn\rho_{-1}(\sigma)}{dn\mu_{-1-1}}&=\exp\bc{dn(H(\mu)-H(\rho(\sigma))+O(\log n)}.
	\end{align}
	Finally, combining \Lem~\ref{Lemma_mu} with \eqref{eqCor_mu2} and \eqref{eqCor_mu3}, we obtain
	\begin{align*}
		\pr\brk{\G\in\cG(\sigma,\mu)}&=\exp\bc{dn(H(\mu)-2H(\rho(\sigma)))/2+O(\log n)}
		=\exp\bc{dn(H(\mu)-H(\rho(\sigma)\tensor\rho(\sigma)))/2+O(\log n)}\\
									 &=
									 \exp\bc{-dn\KL\mu{\rho(\sigma)\tensor\rho(\sigma)}/2+O(\log n)},
	\end{align*}
	as claimed.
\end{proof}

Let $\cM_n=\bigcup_{\sigma\in\{\pm1\}^{V_n}}\cM(\sigma)$ be the set of all conceivable distributions $\mu$.
Moreover, for $\mu\in\cM_n$ set $\rho_1(\mu)=\mu_{11}+\mu_{1-1}$ and $\rho_{-1}(\mu)=1-\rho_1(\mu)$.
Additionally, let $\mu^*=\mu^*_{\b}$ be the distribution
\begin{align}\label{eqvarphimax}
	\mu^*_{11}&=\mu^*_{-1-1}=\frac{1}{2(1+\eul^{\b})},&
	\mu^*_{1-1}&=\mu^*_{-11}=\frac{\eul^{\b}}{2(1+\eul^{\b})}.
\end{align}
Furthermore, let $\cM_n^*$ be the set of all $\mu\in\cM_n$ such that $\dTV(\mu,\mu^*)<n^{-0.49}$.
Finally, let $\cG(\mu)$ be the set of all pairs $(G,\sigma)$ such that $\sigma\in\{\pm1\}^{V_n}$ satisfies $\rho_1(\sigma)=\rho_1(\mu)$ and $G\in\cG(\sigma,\mu)$.
The following lemma supplies the promised formula for the first moment of $Z_{\G,\b}$.

\begin{lemma}\label{Lemma_mean}
For all $d\geq3,\b>0$ we have 
\begin{align}\label{eqLemma_mean}
	\ex[Z_{\G,\b}]&=(1+\exp(-n^{\Omega(1)}))\sum_{\mu\in\cM_n^*}|\cG(\mu)|\exp\bc{-\frac{dn}{2}\bc{\mu_{11}+\mu_{-1-1}}}
	=\Theta\bc{2^n\bcfr{1+\eul^{-\b}}{2}^{dn/2}}.
\end{align}
\end{lemma}
\begin{proof}
	For a given $\mu\in\cM_n$ the total number of $\sigma\in\{\pm1\}^{V_n}$ with $\mu\in\cM(\sigma)$ equals $\binom n{\rho_1(\mu)n}$.
	Therefore, \Cor~\ref{Cor_mu} and \eqref{eqStirling} yield
	\begin{align}\label{eqLemma_mean0}
		\ex[Z_{\G,\b}]&=\sum_{\mu\in\cM_n}|\cG(\mu)|\exp\bc{-\frac{dn}{2}\bc{\mu_{11}+\mu_{-1-1}}}\nonumber\\
					&=\sum_{\mu\in\cM_n}\binom{n}{\rho_1(\sigma)n}\exp\bc{-\frac{dn}{2}\brk{\KL{\mu}{\rho(\mu)\tensor\rho(\mu)}+\b\bc{\mu_{11}+\mu_{-1-1}}}+O(\log n)}\\
					  &=\max_{\mu\in\cM_n}\exp\bc{n\brk{H(\rho(\mu))-\frac{d}{2}\KL{\mu}{\rho(\mu)\tensor\rho(\mu)}-\frac{d\b}{2}\bc{\mu_{11}+\mu_{-1-1}}}+O(\log n)}.
		\label{eqLemma_mean1}
	\end{align}
	Due to the linear relations \eqref{eqmuconstr} we can view the expression inside the square brackets, i.e.,
	\begin{align}\label{eqvarphi}
		\varphi_{d,\b}(\mu)&=H(\rho(\mu))-\frac{d}{2}\KL{\mu}{\rho(\mu)\tensor\rho(\mu)}-\frac{d\b}2\bc{\mu_{11}+\mu_{-1-1}},
	\end{align}
	 as a function of the two variables $\mu_{11}$ and $\mu_{-1-1}$.
	 The function is strictly concave because the entropy function is strictly concave and the Kullback-Leibler divergence is convex.
	 Hence, the unique stationary point of $\varphi_{d,\b}$ is its maximiser.
	 Since the derivatives of $\varphi_{d,\b}$ work out to be
	 \begin{align*}
		 \frac{\partial\varphi_{d,\b}}{\partial\mu_{11}}&=\frac{d-1}{2}\log\frac{\rho_1(\mu)}{\rho_{-1}(\mu)}+\frac{d}{2}\log\frac{\mu_{1-1}}{\mu_{11}}-\frac{d\b}{2},&
		 \frac{\partial\varphi_{d,\b}}{\partial\mu_{-1-1}}&=\frac{1-d}{2}\log\frac{\rho_1(\mu)}{\rho_{-1}(\mu)}+\frac{d}{2}\log\frac{\mu_{1-1}}{\mu_{-1-1}}-\frac{d\b}{2},
	 \end{align*}
	 the stationary point occurs at $\mu^*$.
	 Substituting the solution \eqref{eqvarphimax} into \eqref{eqLemma_mean1}, we obtain
	 \begin{align}\label{eqLemma_mean2}
		 \ex[Z_{\G,\b}]&=\exp\bc{n\brk{\log 2+\frac{d}{2}\log\frac{1+\eul^{-\b}}{2}}+O(\log n)}.
	 \end{align}
	 as well as the first equality sign in~\eqref{eqLemma_mean}.
	 To obtain the second part of~\eqref{eqLemma_mean} we take another look at \Lem~\ref{Lemma_mu}, which shows together with Stirling's formula that there exists $c=c(d,\b)$ such that
	 \begin{align}\label{eqLemma_mean7}
		 |\cG(\mu)|\frac{\exp\bc{-\frac{\b dn}{2}(\mu_{11}+\mu_{-1-1})}}{(dn-1)!!}&=\frac{c}{n}\exp\bc{-n\varphi_{d,\b}(\mu)}\qquad\mbox{uniformly for all }\mu\in\cM_n^*.
	 \end{align}
	 Since the function $\varphi_{d,\b}$ is strictly concave, \eqref{eqLemma_mean7} shows together with the first part of \eqref{eqLemma_mean} and the Laplace method that
	 \begin{align*}
		 \ex[Z_{\G,\b}]&=\Theta(\exp(-n\varphi_{d,\b}(\mu^*)))=\Theta\bc{2^n\bcfr{1+\eul^{-\b}}{2}^{dn/2}},
	 \end{align*}
	 which completes the proof.
\end{proof}

Having calculated $\ex[Z_{\G,\b}]$ sufficiently accurately, we proceed to extend this formula to the simple random graph $\GG$ and to the truncated first moment $\ex[Z_{\GG,\b}\vecone\cbc{\cO}]$.
Fortunately we can kill these two birds with one stone.

\subsection{The truncated first moment}\label{sec_trunc1}
We need to calculate truncated first moments of the form $\ex[Z_{\G,\b}\vecone\cA]$ for some event $\cA$.
To this end we define a pairing model variant of the stochastic block model.
In analogy to \eqref{eqsbm} we draw $\vec\sigma^*\in\{\pm1\}^{V_n}$ uniformly at random.
Further, given $\vec\sigma^*$ for any possible outcome $G$ of $\G$ we let
\begin{align}\label{eqG*}
	\pr\brk{\G^*=G\mid\vec\sigma^*}&\propto\exp(-\b\cH_G(\vec\sigma^*)).
\end{align}
The following lemma will enable us to reduce the task of computing $\ex[Z_{\G,\b}\vecone\cA]$ for an event $\cA$ to estimating the probability of $\G^*\in\cA$.
Similar lemmas have been known for other random problems since the work of Achlioptas and Coja-Oghlan~\cite{Barriers}.

\begin{lemma}\label{Lemma_null}
	Let $d\geq3,\b>0$ and let $\cE$ be a set of graph/spin configuration pairs. Then
	$$\frac{1}{\ex[Z_{\G,\b}]}\sum_{\sigma\in\{\pm1\}^{V_n}}\ex[\psi_{\G,\b}(\sigma)\vecone\{(\G,\sigma)\in\cE\}]=\Theta(\pr[(\G^*,\vec\sigma^*)\in\cE])+o(1).$$
\end{lemma}
\begin{proof}
	The definition \eqref{eqG*} of $\G^*$ ensures that
\begin{align}\label{eqLemma_null1}
	\sum_{\sigma\in\{\pm1\}^{V_n}}\ex[\psi_{\G,\b}(\sigma)\vecone\{(\G,\sigma)\in\cE\}]&=
	\sum_{\sigma\in\{\pm1\}^{V_n}}{\pr[(\G^*,\vec\sigma^*)\in\cE\mid\vec\sigma^*=\sigma]}{\ex[\psi_{\G,\b}(\sigma)]}.
\end{align}
We now split the above sums up into three parts: for a small $\eps>0$ pick $C>0$ large and let
\begin{align*}
	S&=\cbc{\sigma\in\{\pm1\}^{V_n}:|\rho_1(\sigma)-1/2|\leq Cn^{-1/2}},&
	S'&=\cbc{\sigma\in\{\pm1\}^{V_n}\setminus S:|\rho_1(\sigma)-1/2|\leq n^{-0.49}},&
	S''&=\{\pm1\}^{V_n}\setminus(S\cup S').
\end{align*}
Then \Lem~\ref{Lemma_mean} implies that
\begin{align}\label{eqLemma_null2}
	\sum_{\sigma\in S''}\ex[\psi_{\G,\b}(\sigma)]&=o(\Erw[Z_{\G,\b}]).
\end{align}
In fact, \eqref{eqLemma_mean7} implies that for large enough $C$,
\begin{align}\label{eqLemma_null5}
	\sum_{\sigma\in S'}\ex[\psi_{\G,\b}(\sigma)]&\leq\eps\Erw[Z_{\G,\b}].
\end{align}
In addition, \eqref{eqLemma_mean7} implies together with the fact that $\mu^*$ is the unique stationary point of the concave function $\varphi_{d,\b}$ that
\begin{align}\label{eqLemma_null4}
	\ex[\psi_{\G,\b}(\sigma)]&=O(2^{-n}\ex[Z_{\G,\b}])\qquad\mbox{uniformly for all }\sigma\in\{\pm1\}^{V_n},\\
	\ex[\psi_{\G,\b}(\sigma)]&=\Theta(2^{-n}\ex[Z_{\G,\b}])\qquad\mbox{uniformly for all }\sigma\in S.\label{eqLemma_null6}
\end{align}
Furthermore, because $\vec\sigma^*$ is uniformly random we can choose $C$ so large that
\begin{align}\label{eqLemma_null3}
	\sum_{\sigma\in\{\pm1\}^{V_n}}{\pr[(\G^*,\vec\sigma^*)\in\cE\mid\vec\sigma^*=\sigma]}\leq\eps+\sum_{\sigma\in S}{\pr[(\G^*,\vec\sigma^*)\in\cE\mid\vec\sigma^*=\sigma]}.
\end{align}
Combining \eqref{eqLemma_null1}--\eqref{eqLemma_null3} and taking $\eps\to0$ slowly, we obtain the assertion.
\end{proof}

\noindent
As an immediate consequence of \Lem~\ref{Lemma_null} we obtain the following.

\begin{corollary}\label{Cor_null}
	For all $d\geq3,\b>0$ and for any event $\cA$ the following two statements are true.
	\begin{enumerate}[(i)]
		\item If $\pr\brk{\G^*\in\cA}=\Omega(1)$, then $\ex[Z_{\G,\b}\vecone\cA]=\Theta(\ex[Z_{\G,\b}])$.
		\item We have $\pr\brk{\G^*\in\cA}=1-o(1)$ iff $\ex[Z_{\G,\b}\vecone\cA]\sim\ex[Z_{\G,\b}]$.
	\end{enumerate}
\end{corollary}

As an application of \Cor~\ref{Cor_null} we will compute $\ex[Z_{\GG,\b}]=\ex[Z_{\G,\b}\vecone\cS]$.
To this end we need bound the probability of the event $\G^*\in\cS$ away from zero.

\begin{lemma}\label{Lemma_simple*}
	For all $d\geq3,\b>0$ we have $\pr\brk{\G^*\in\cS}=\Omega(1)$.
\end{lemma}
\begin{proof}
	Following the well known proof that $\pr[\G\in\cS]=\Omega(1)$, we will use the method of moments.
	Thus, fix $\mu\in\cM_n^*$ and $\sigma\in\{\pm1\}^{V_n}$ with $\rho_1(\sigma)=\rho_1(\mu)$.
	Let $X$ be the number of self-loops of $\G$ and let $Y$ be the number of double-edges.
	We will show that for any fixed integers $k,\ell\geq1$,
	\begin{align}\label{eqLemma_simple*1}
		\ex\brk{\prod_{j=1}^k (X-j+1)\prod_{j=1}^\ell (Y-j+1)}&\sim\kappa^k\lambda^\ell\qquad\mbox{with }\kappa=\frac{d-1}{\eul^{\b}+1},\ \lambda=\frac{(d-1)^2(1+\eul^{2\b})}{2(1+\eul^{\b})^2}.
	\end{align}
	Clearly \eqref{eqLemma_simple*1} implies that $\pr[\G^*\in\cS]=\pr[X=Y=0]\sim\exp(-\kappa-\lambda)=\Omega(1)$.

To verify \eqref{eqLemma_simple*1} we start by computing the means of $X,Y$.
To be more precise, let $X_1$ be the number of self-loops at vertices $v_i$ with $\sigma_{v_i}=1$.
In order to construct a self-loop we need to pick a vertex and two of its clones and calculate the probability that these clones get matched.
Thus, the number of choices equals $\binom d2\rho_1(\sigma)n$.
Therefore, \Lem~\ref{Lemma_mu} and \eqref{eqvarphimax} yield
\begin{align}\label{eqLemma_simple*5}
	\ex[X_1\mid\cG(\sigma,\mu)]&=\frac{\binom d2\rho_1(\sigma)n\binom{dn\rho_1(\sigma)-2}{dn\mu_{11}-2}\binom{dn\rho_{-1}(\sigma)}{dn\mu_{-1-1}}(dn\mu_{11}-3)!!(dn\mu_{-1-1}-1)!!(dn\mu_{1-1})!}{\binom{dn\rho_1(\sigma)}{dn\mu_{11}}\binom{dn\rho_{-1}(\sigma)}{dn\mu_{-1-1}}(dn\mu_{11}-1)!!(dn\mu_{-1-1}-1)!!(dn\mu_{1-1})!}\sim\frac{\kappa}{2}.
	\end{align}
	Because \eqref{eqvarphimax} ensures that $\rho_1(\sigma)\sim1/2$, \eqref{eqLemma_simple*5} implies that
\begin{align}\label{eqLemma_simple*6}
	\ex[X\mid\cG(\sigma,\mu)]&\sim\kappa.
	\end{align}

	Similar considerations yield the mean of $Y$.
	Specifically, we decompose $Y$ into $Y_{11}$, $Y_{-1-1}$ and $Y_{1-1}$, which, respectively, count double-edges among vertices assigned spin $1$, among vertices with spin $-1$, and between vertices with different spins.
	To work out $Y_{11}$ we need to select two vertices with spin $1$, two clones of each and a perfect matching.
	Thus, the number of choices comes to $2\binom{\rho_1(\sigma)n}2\binom{d}2^2$.
Hence, \Lem~\ref{Lemma_mu} and \eqref{eqvarphimax} yield
\begin{align}
	\ex[Y_{11}\mid\cG(\sigma,\mu)]&=\frac{2\binom{\rho_1(\sigma)n}2\binom{d}2^2\binom{dn\rho_1(\sigma)-4}{dn\mu_{11}-4}\binom{dn\rho_{-1}(\sigma)}{dn\mu_{-1-1}}(dn\mu_{11}-5)!!(dn\mu_{-1-1}-1)!!(dn\mu_{1-1})!}{\binom{dn\rho_1(\sigma)}{dn\mu_{11}}\binom{dn\rho_{-1}(\sigma)}{dn\mu_{-1-1}}(dn\mu_{11}-1)!!(dn\mu_{-1-1}-1)!!(dn\mu_{1-1})!}\nonumber\\
								  &	\sim \frac{(d-1)^2\mu_{11}^2}{4\rho_1(\sigma)^2}\sim\frac{(d-1)^2}{4(\eul^{\b}+1)^2}.
			\label{eqLemma_simple*7}
	\end{align}
	The same calculation applies to the mean of $Y_{-1-1}$.
	Moreover, analogously we obtain
\begin{align}
	\ex[Y_{1-1}\mid\cG(\sigma,\mu)]&=\frac{2\rho_1(\sigma)\rho_{-1}(\sigma)n^2\binom{d}2^2\binom{dn\rho_1(\sigma)-2}{dn\mu_{11}}\binom{dn\rho_{-1}(\sigma)-2}{dn\mu_{-1-1}}(dn\mu_{11}-1)!!(dn\mu_{-1-1}-1)!!(dn\mu_{1-1}-2)!}{\binom{dn\rho_1(\sigma)}{dn\mu_{11}}\binom{dn\rho_{-1}(\sigma)}{dn\mu_{-1-1}}(dn\mu_{11}-1)!!(dn\mu_{-1-1}-1)!!(dn\mu_{1-1})!}\nonumber\\
								   &\sim\frac{(d-1)^2\mu_{1-1}^2}{2\rho_1\rho_{-1}}=\frac{(d-1)^2\eul^{2\b}}{2(1+\eul^{\b})^2}.
								   \label{eqLemma_simple*8}
	\end{align}
	Combining \eqref{eqLemma_simple*7} and  \eqref{eqLemma_simple*8}, we obtain
	\begin{align}\label{eqLemma_simple*9}
			\ex[Y\mid\cG(\sigma,\mu)]&\sim\lambda.
				\end{align}
	
				The calculations that we performed towards \eqref{eqLemma_simple*6} and \eqref{eqLemma_simple*9} easily extend to a proof of \eqref{eqLemma_simple*1}.
				Indeed, instead of just accounting for the choice of placing a single double-edge or loop, we need to place fixed numbers $k,\ell$.
				Since $k,\ell$ remain bounded as $n\to\infty$, the probability that any choices overlap is $O(1/n)$.
				Therefore, the joint factorial moment of $X,Y$ works out to be $\kappa^k\lambda^\ell$, which is  \eqref{eqLemma_simple*1}.
\end{proof}

\begin{proof}[Proof of \Lem~\ref{Cor_simple*}]
The lemma follows from \Lem~\ref{Lemma_mean}, \Cor~\ref{Cor_null} and \Lem~\ref{Lemma_simple*}.
\end{proof}

\begin{proof}[Proof of \Lem~\ref{Lemma_hunch}]
This is an immediate consequence of \Cor~\ref{Cor_null} and \Lem~\ref{Lemma_simple*}.
\end{proof}

\subsection{Coupling with the broadcasting process}\label{Sec_couple}
In this section we are going to establish a coupling of the local structure of $\G^*$ around a given vertex $v_i$ with the broadcasting process from \Lem~\ref{lem_reconstruction}.
Specifically, we are going to prove the following statement.

\begin{lemma}\label{Cor_O}
For any $d\geq3,\b>0$ there exists $\eps_n=o(1)$ such that the event $\cO$ from \eqref{eqOdef} satisfies $ \ex[Z_{\G,\b}\vecone\cbc{\cO}]\sim\ex[Z_{\G,\b}].  $
\end{lemma}

We begin the proof of of \Lem~\ref{Cor_O} by showing that the bounded-depth neighbourhoods in $\G^*$ are typically acyclic.

\begin{lemma}\label{Lemma_acyclic}
Let $d\geq3$, $\b>0$.
Moreover, for an integer $\ell\geq1$ let $C_\ell$ be the number of cycles of length $\ell$ in $\G^*$.
Then for any fixed integer $L$ we have $\sum_{\ell\leq L}C_\ell=O(\log n)$ \whp
\end{lemma}
\begin{proof}
By \Lem~\ref{Lemma_mean} and \Cor~\ref{Cor_null} we may condition on the event $\cG(\mu)$ for some $\mu\in\cM_n^*$ and on the event $|\rho_1(\vec\sigma^*)|\sim 1/2$.
A cycle of length $\ell$ passes through (not necessarily distinct) vertices $\vec u=(u_1,\ldots,u_\ell)$.
For each step of the cycle we select a clone $i_t$ where the cycle enters and one $j_t\neq i_t$ where it leaves.
Set $\vec i=(\vec i_1,\ldots,\vec i_t)$ and $\vec j=(\vec j_1,\ldots,\vec j_t)$.
However, we overcounted by a factor of $2\ell$ (for the direction and the choice of the starting point).
Given these choices let $e_{11}$ be the number of edges of the cycle that connect two vertices of spin $1$ under $\vec\sigma^*$ and define $e_{-1-1}$ similarly.
Moreover, let $e_{1-1}$ be the number of cycle edges that join vertices with different spins.
Following \Lem~\ref{Lemma_mu} we estimate the probability of the event $\cC(\vec u,\vec i,\vec j)$ that the specified cycle actually appears in $\G^*$ by
\begin{align}
\pr\brk{\cC(\vec u,\vec i,\vec j)\mid\cG(\mu),\vec\sigma^*}&
\sim{\binom{dn\rho_1(\sigma)-2e_{11}-e_{1-1}}{dn\mu_{11}-2e_{11}}\binom{dn\rho_{-1}(\sigma)-2e_{-1-1}-e_{1-1}}{dn\mu_{-1-1}-2e_{-1-1}}}{\binom{dn\rho_1(\sigma)}{dn\mu_{11}}^{-1}\binom{dn\rho_{-1}(\sigma)}{dn\mu_{-1-1}}^{-1}}\nonumber\\
	&\qquad\cdot\frac{(dn\mu_{11}-2e_{11}-1)!!(dn\mu_{-1-1}-2e_{-1-1}-1)!!(dn\mu_{1-1}-e_{1-1})!}{(dn\mu_{11}-1)!!(dn\mu_{-1-1}-1)!!dn\mu_{1-1}!}\nonumber\\
&\sim (dn)^{-\ell}\bcfr{\mu_{11}}{\rho_1(\vec\sigma^*)^2}^{e_{11}}\bcfr{\mu_{-1-1}}{\rho_{-1}(\vec\sigma^*)^2}^{e_{-1-1}}\bcfr{\mu_{1-1}}{\rho_1(\vec\sigma^*)\rho_{-1}(\vec\sigma^*)}^{e_{1-1}}\nonumber\\
&\sim \bcfr{2}{dn(\eul^{\b}+1)}^\ell\eul^{\b e_{1-1}}\qquad[\mbox{due to \eqref{eqvarphimax}}].
\label{eqLemma_acyclic1}
\end{align}
Since the total number of choices for $\vec u,\vec i,\vec j$ is bounded by $n^\ell\binom d2^\ell$, \eqref{eqLemma_acyclic1} implies that $\ex[C_\ell\mid\cG(\mu),\vec\sigma^*]=O(1)$.
Therefore, the assertion follows from Markov's inequality.
\end{proof}

For a vertex $v$ of $\G^*$ and an integer $\ell\geq0$ let $\vec\sigma^*_{v,\ell}$ be the spin configuration that $\vec\sigma^*$ induces on the vertices at distance at most $\ell$ from $v$.
Furthermore, let $\vec\tau_\ell,\vec\tau_\ell'$ be two independent copies of the spin configuration that the broadcasting process from \Sec~\ref{Sec_bnr} induces on the vertices of the infinite $d$-regular tree $\TT_d$ at distance at most $\ell$ from its root.

\begin{lemma}\label{Lemma_broadcast_coupling}
For any $d\geq3,\b>0,\ell\geq0$ the spin configurations $\vec\sigma^*_{v_1,\ell}$ and $\vec\tau_\ell$ have total variation distance $o(1)$.
\end{lemma}
\begin{proof}
Thanks to \Lem~\ref{Lemma_mean} and \Cor~\ref{Cor_null} we may condition on the event $\cG(\mu)$ for a $\mu\in\cM_n^*$ and on $|\rho_1(\vec\sigma^*)|\sim 1/2$.
Moreover, due to \Lem~\ref{Lemma_acyclic} we may confine ourselves to the case that the depth-$\ell$ neighbourhood of $v_1$ is acyclic.
Let $T$ be a possible outcome of the depth-$\ell$ neighbourhood of $v_1$ under these assumptions.
Moreover, let $e_{11},e_{-1-1},e_{1-1}$ be the numbers of edges of $T$ that join vertices both assigned spin $1$ under $\vec\sigma^*$, or both assigned spin $-1$, or assigned different spins, respectively.
Further, set $e=e_{11}+e_{1-1}+e_{-1-1}$.
Finally, let $\cE(T)$ be the event that $T$ occurs in $\G^*$.
Then \Lem~\ref{Lemma_mu} and \eqref{eqvarphimax} show that
\begin{align}
		\pr\brk{\cE(T)\mid\cG(\mu),\vec\sigma^*}
			&=\binom{dn\rho_1(\vec\sigma^*)-2e_{11}-e_{1-1}}{dn\mu_{11}-2e_{11}}\binom{dn\rho_1(\vec\sigma^*)}{dn\mu_{11}}^{-1}\nonumber\\
					&\quad\cdot\binom{dn\rho_{-1}(\vec\sigma^*)-2e_{-1-1}-e_{1-1}}{dn\mu_{-1-1}-2e_{-1-1}}
					\binom{dn\rho_{-1}(\vec\sigma^*)}{dn\mu_{-1-1}}^{-1}\nonumber\\
			&\quad\cdot\frac{(dn\mu_{11}-2e_{11}-1)!!(dn\mu_{-1-1}-2e_{-1-1}-1)!!(dn\mu_{1-1}-e_{1-1})!}{(dn\mu_{11}-1)!!(dn\mu_{-1-1}-1)!!(dn\mu_{1-1})!}\nonumber\\
	&\sim\bcfr{2}{dn}^{2e}\bcfr{\eul^{\b}}{1+\eul^{\b}}^{e_{1-1}}\bcfr{1}{1+\eul^{\b}}^{e_{11}+e_{-1-1}}.
\label{eqLemma_broadcast_coupling1}
	\end{align}
Hence, \eqref{eqLemma_broadcast_coupling1} shows that the probability of observing
a given 
spin assignment $\vec\sigma_{v_1,\ell}^*$ 
depends only on the number of edges
joining vertices with the same spin, and that this dependence is precisely the same as in the case \eqref{eqbroadcast} of the broadcasting process.
Thus, $\vec\sigma_{v_1,\ell}^*$ and $\vec\tau_\ell$ have total variation distance $o(1)$.
\end{proof}

For a graph $G$, a vertex $v$ of $G$ and an integer $\ell>0$ let $\partial^\ell(G,v)$ be the set of vertices at distance precisely $\ell$ from $v$.
Further, for a spin configuration $\chi\in\{\pm1\}^{V(G)}$ let
\begin{align*}
\mu_{G,\b,v,\ell}(s\mid\chi)&=\frac{\sum_{\sigma\in\{\pm1\}^{V(G)}}\vecone\{\sigma_v=s,\,\forall u\in\partial^\ell(G,v):\sigma_u=\chi_u\}\exp(-\b\cH_{G}(\sigma))}{\sum_{\sigma\in\{\pm1\}^{V(G)}}\vecone\{\forall u\in\partial^\ell(G,v):\sigma_u=\chi_u\}\exp(-\b\cH_{G}(\sigma))}\qquad(s=\pm1);
\end{align*}
in words, this is the conditional Boltzmann marginal of $v$ given the `boundary condition' $\chi$ at the vertices at distance precisely $\ell$ from $v$.

\begin{corollary}\label{Cor_broadcast_coupling}
For any $d\geq3,\b>0,\eps>0$ there exists $\ell>0$ such that $\ex\sum_{i=1}^n\abs{\mu_{\G^*,\b,v_i,\ell}(1|\vec\sigma^*)-\frac{1}{2}}<\eps n$.
\end{corollary}
\begin{proof}
Because the random pair $(\G^*,\vec\sigma^*)$ is invariant under vertex permutations, we have
\begin{align*}
\ex\sum_{i=1}^n\abs{\mu_{\G^*,\b,v_i,\ell}(1|\vec\sigma^*)-\frac{1}{2}}=n\ex\abs{\mu_{\G^*,\b,v_i,\ell}(1|\vec\sigma^*)-\frac{1}{2}} 
\end{align*}
and \Lem~\ref{lem_reconstruction} and \Lem~\ref{Lemma_broadcast_coupling} show that the r.h.s.\ gets small in the limit of large $\ell$.
\end{proof}

\begin{proof}[Proof of \Lem~\ref{Cor_O}]
We apply \Cor~\ref{Cor_broadcast_coupling} to a function $\eps'_n=o(1)$ that tends to zero slowly.
Specifically, let $X(\G^*,\vec\sigma^*)=\sum_{i=1}^n\vecone\{|\mu_{\G^*,\b,v_i,\ell}(1|\vec\sigma^*)-\frac{1}{2}|>\eps'\}$.
\Cor~\ref{Cor_broadcast_coupling} implies together with Markov's inequality that $\pr[X(\G^*,\vec\sigma^*)>\eps''n]\leq\eps''$ for a suitable $1\ll\ell=o(\log n)$, provided that $\eps',\eps''=o(1)$ tend to zero slowly enough.
Hence, \Lem~\ref{Lemma_null} shows that
\begin{align}\label{eqCor_O0}
\ex[Z_{\GG,\b}\pr[X(\GG,\vec\sigma_{\GG})\leq\eps''n\mid\GG]]\sim\ex[Z_{\GG,\b}].
\end{align}
We claim that \eqref{eqCor_O0} implies that there exists $n^{-1/4}\ll\delta=o(1)$ such that
\begin{align}\label{eqCor_O1}
\ex[Z_{\GG,\b}\vecone\{\pr[X(\GG,\vec\sigma_{\GG})>\delta n\mid\GG]<\delta \}]\sim\ex[Z_{\GG,\b}].
\end{align}
Indeed, \eqref{eqCor_O0} shows that for a suitable $\delta $,
\begin{align*}
\ex[Z_{\GG,\b}\vecone\{\pr[X(\GG,\vec\sigma_{\GG})>\delta n\mid\GG]\geq\delta \}] &\leq\delta^{-1} \ex[Z_{\GG,\b}\pr[X(\GG,\vec\sigma_{\GG})>\eps''n\mid\GG]]=o(\ex[Z_{\GG,\b}]).
\end{align*}

Due to \eqref{eqCor_O1} it suffices to prove that there exists $\eps=o(1)$ such that for any $d$-regular graph $G$ of sufficiently large order $n$ the following is true:
\begin{align}\label{eqCor_O2}
\mbox{if $\pr[X(G,\vec\sigma_{G})>\delta n]<\delta$ then }
\ex\abs{\vec\sigma_{G}\cdot\vec\sigma_{G}'}\leq\eps n.
\end{align}
Indeed, since $\ex|\vec\sigma_{G}\cdot\vec\sigma_{G}'|=\ex[\ex[|\vec\sigma_{G}\cdot\vec\sigma_{G}'|\mid\vec\sigma_G']]$, we may condition on $\vec\sigma_G'$.
Hence, let $V_1'$ contain all vertices $v\in V(G)$ such that $\vec\sigma'_{G,v}=1$ and let $V_{-1}'=V(G)\setminus V_1'$.
Further, let
\begin{align*}
Y_s&=\sum_{v\in V_s'}\vecone\{\vec\sigma_{G,v}=1\}\qquad(s=\pm1).
\end{align*}
Then to establish \eqref{eqCor_O2} we just need to prove that
\begin{align}\label{eqCor_O3}
|Y_s-|V_s'|/2|&=o(n)\qquad\mbox{ \whp\ for $s=\pm1$.}
\end{align}
 
To deduce \eqref{eqCor_O3} fix $s=\pm1$.
If $|V_s'|<\delta^{1/3}n$, say, then \eqref{eqCor_O3} is immediate.
Hence, we may assume that $|V_s'|\geq\delta^{1/3}n$.
Draw a vertex $\vec v\in V_s'$ uniformly at random, independently of $\vec\sigma_{G}$.
Then the assumption $\pr[X(G,\vec\sigma_{G})\leq\delta n\mid G]$ implies that
\begin{align}\label{eqCor_O4}
\pr\brk{|\mu_{ G,\b,\vec v,\ell}(1|\vec\sigma_{ G})-1/2|>\eps'\mid \vec\sigma'_{ G}}&<\delta^{2/3}.
\end{align}
Now, consider a spin configuration $\vec\sigma_{ G}''$ drawn from the Boltzmann distribution given the event
\begin{align*}
\cA(\vec v,\vec\sigma_{ G})&=\cbc{\sigma\in\{\pm1\}^{V(G)}:\sigma_w=\vec\sigma_{ G,w}\mbox{ for all $w$ at distance $\ell$ or more from $\vec v$}}.
\end{align*}
In other words, $\vec\sigma_{ G}''$ is obtained by re-sampling the spins of the vertices at distance less than $\ell$ from $\vec v$ from the Boltzmann distribution with the boundary condition that $\vec\sigma_{ G}$ induces on the vertices at distance precisely $\ell$ from $\vec v$.
Since $\vec\sigma_{ G}$ is a sample from $\mu_{ G,\b}$, so is $\vec\sigma_{ G}''$.
Moreover, $\vec\sigma_{ G}''$ is independent of $\vec\sigma_{ G}'$.
Therefore, \eqref{eqCor_O4} yields
\begin{align}\label{eqCor_O5}
\ex[Y_s\mid \vec\sigma_{ G}']&=|V_s'|\pr\brk{\vec\sigma_{ G,\vec v}''=1\mid \vec\sigma_{G}'}=|V_s'|\ex[\mu_{G,\b,\vec v,\ell}(1|\vec\sigma_{G})\mid G,\vec\sigma'_{G}]\sim|V_s'|/2.
\end{align}

To complete the proof we apply similarly reasoning to estimate $\ex[Y_s^2\mid G,\vec\sigma_{ G}']$.
Specifically, let $\vec v'$ be a second vertex drawn uniformly from $V_s'$, independently of $\vec v$.
Then in analogy to \eqref{eqCor_O4} we obtain
\begin{align}\label{eqCor_O6}
\pr\brk{|\mu_{ G,\b,\vec v,\ell}(1|\vec\sigma_{ G})-1/2|>\eps'\vee |\mu_{ G,\b,\vec v',\ell}(1|\vec\sigma_{ G})-1/2|>\eps' \mid \vec\sigma'_{ G}}&<2\delta^{2/3}.
\end{align}
Further, draw $\vec\sigma_{ G}'''$ from the Boltzmann distribution given
\begin{align*}
\cA(\vec v,\vec v',\vec\sigma_{ G})&=\cbc{\sigma\in\{\pm1\}^{V(G)}:\sigma_w=\vec\sigma_{ G,w}\mbox{ for all $w$ at distance $\ell$ or more from both $\vec v,\vec v'$}}.
\end{align*}
Since $G$ is $d$-regular, $\ell$ is fixed and $|V_s'|\geq\delta^{1/3}n\geq\sqrt n$, the vertices $\vec v,\vec v'$ are distance more than $2\ell$ apart \whp\
In this case the spins $\vec\sigma_{\vec v}''',\vec\sigma_{\vec v'}'''$ are conditionally independent given $\vec\sigma_G'$.
Consequently, \eqref{eqCor_O6} implies that
\begin{align}\nonumber
\ex[Y_s^2\mid G,\vec\sigma_{ G}']&=|V_s'|^2\pr\brk{\vec\sigma_{ G,\vec v}'''=1,\vec\sigma_{ G,\vec v'}'''=1\mid G,\vec\sigma_{ G}'}\\&
		=|V_s'|^2\ex[\mu_{ G,\b,\vec v,\ell}(1|\vec\sigma_{ G})\mid G,\vec\sigma'_{ G}]\ex[\mu_{ G,\b,\vec v',\ell}(1|\vec\sigma_{ G})\mid G,\vec\sigma'_{ G}]+o(2^2)=|V_s'|^2/4+o(n^2).
\label{eqCor_O7}
\end{align}
Finally, combining \eqref{eqCor_O5}, \eqref{eqCor_O7} and Chebyshev's inequality, we obtain \eqref{eqCor_O3}, completing the proof.
\end{proof}

\begin{proof}[Proof of \Lem~\ref{cor_reconstruction}]
The lemma follows directly from \Cor~\ref{Cor_null}, \Lem~\ref{Lemma_simple*} and \Lem~\ref{Cor_O}.
\end{proof}

\subsection{The truncated second moment}\label{Sec_smm}
The aim in this section is to show the following.

\begin{lemma}\label{Lemma_O2}
For any $d\geq3,\b>0$ there exists $\eps_n=o(1)$ such that the event $\cO$ from \eqref{eqOdef} satisfies $$\ex\brk{Z_{\GG(n,d),\b}^2\vecone\cbc{\cO}}\leq \ex\brk{Z_{\GG(n,d),\b}}^2\exp(o(n)).$$
\end{lemma}

Toward the proof of \Lem~\ref{Lemma_O2} we require the following observation.

\begin{fact}\label{Fact_perp}
Suppose that $(\mu_n)_{n\geq1}$ is a sequence of probability measures $\mu_n\in\cP(\{\pm1\}^n)$ such that
\begin{align*}
\lim_{n\to\infty}\sum_{\sigma,\sigma'\in\{\pm1\}^n}\frac{|\sigma\cdot\sigma'|}n\mu_n(\sigma)\mu_n(\sigma')=0.
\end{align*}
Then $\lim_{n\to\infty}n^{-1}\sum_{\sigma\in\{\pm1\}^n}|\sigma\cdot\vecone|\mu_n(\sigma)\mu_n(\sigma')=0.$
\end{fact}

\begin{corollary}\label{Lemma_sm_trunc}
For any $d\geq3,\b>0$ there exists $\delta=\delta_n=o(1)$ such that 
\begin{align*}
\ex\brk{Z_{\GG(n,d),\b}^2\vecone\cbc{\cO}}&\leq(1+o(1))\sum_{\sigma,\sigma'\in\{\pm1\}^{V_n}}\vecone\cbc{|\sigma\cdot\vecone|,|\sigma'\cdot\vecone|,|\sigma\cdot\sigma'|\leq\delta n}\ex[\exp(-\b\cH_{\GG}(\sigma)-\b\cH_{\GG}(\sigma'))].
\end{align*}
\end{corollary}
\begin{proof}
This is an immediate consequence of Fact~\ref{Fact_perp} and the definition \eqref{eqOdef} of the event $\cO$.
\end{proof}

\begin{lemma}\label{Lemma_Philipp}
For any $d\geq3,\b>0$ we have
\begin{align}\label{eqLemma_Philipp}
\sum_{\sigma,\sigma'\in\{\pm1\}^{V_n}}\vecone\cbc{|\sigma\cdot\vecone|,|\sigma'\cdot\vecone|,|\sigma\cdot\sigma'|\leq\delta n}\ex[\exp(-\b\cH_{\G}(\sigma)-\b\cH_{\G}(\sigma'))]
&\leq \exp\bc{nf_d(0,\beta)+O(\delta n)}.
\end{align}
\end{lemma}
\begin{proof}
Given $\sigma,\sigma' \in \pmone^{V_n}$ let $\rho = \rho(\sigma, \sigma') = (\rho_{s,t}(\sigma, \sigma))_{s,t \in \pmone}$ and $\mu = \mu(\sigma, \sigma', \G) = (\mu_{r,s,t,u}(\sigma, \sigma', \G))_{r,s,t,u \in \pmone}$ be the vectors with entries
\begin{align*}
\rho_{s,t}(\sigma,\sigma')&=\frac{1}{n}\sum_{v=1}^n\vecone\cbc{\sigma_v=s,\sigma_w'=t}\qquad(s,t\in\pmone),\\
\vec\mu_{r,s,t,u}(\sigma,\sigma')&=\frac{2}{dn}\sum_{\vwe \in E(\G)}
		\vecone\cbc{\sigma_v=r,\sigma'_v=s,\sigma_w=t,\sigma'_w=u}
		\qquad(r,s,t,u\in\pmone).
\end{align*}
Thus, $\rho(\sigma,\sigma')$ is the empirical distribution of the spin combinations that $\sigma,\sigma'$ assign to the vertices.
Similarly, $\vec\mu$ comprises the statistics of the edges of $\G$ joining vertices with different spin combinations.
Let 
\begin{align}\label{eqLemma_Philipp00}
	\Sigma^\tensor&=\cbc{\sigma,\sigma' \in \pmone^{V_n}:|\sigma\cdot\vecone|,\,|\sigma'\cdot\vecone|,\,|\sigma\cdot\sigma'|\leq\delta n},&
	\cR^\tensor&=\{\rho(\sigma,\sigma'):\sigma,\sigma'\in\Sigma^\tensor\},
\end{align}
and let $\cM^\tensor$ be the set of all possible outcomes of the random vector $\vec\mu(\sigma,\sigma')$ for any $\sigma,\sigma'\in\Sigma^\tensor$.
For $\rho\in\cR^\tensor,\mu\in\cM^\tensor$ we use the shorthands $\rho_{++} = \rho_{+1, +1}$ and $\mu_{++++} = \mu_{+1, +1, +1, +1},$ and similarly for the other possible sign patterns.
Further, let
\begin{align}\label{eqLemma_Philipp0}
\cH(\mu)&=2\bc{\mu_{++++}+\mu_{----}+\mu_{+-+-}+\mu_{-+-+}}+\\
&\qquad\mu_{+++-}+\mu_{++-+}+\mu_{+-++}+\mu_{-+++}+\mu_{---+}+\mu_{--+-}+\mu_{-+--}+\mu_{+---}.
\end{align}
Finally, for $\mu\in\cM^\tensor$ we define $\rho(\mu)\in\cR^\tensor$ by
\begin{align*}
	\rho_{ij}(\mu)&=\sum_{k,l \in \cbc{\pm 1}} \mu_{ijkl}\qquad\mbox{for }i,j\in\cbc{\pm 1}.
\end{align*}

We now claim that for any $\mu\in\cM^\tensor$,
\begin{align}
	\sum_{\sigma,\sigma'\in\{\pm1\}^{V_n}}\pr\brk{\vec\mu(\sigma,\sigma')=\mu}&=\frac{\cX_{\mu}\cY_{\mu}\cZ_{\mu}}{(dn-1)!!}\qquad\mbox{where}\label{eqLemma_Philipp1}\\
	\cX_{\mu} &= \binom{n}{\rho_{++}(\mu)n, \rho_{+-}(\mu)n, \rho_{-+}(\mu)n, \rho_{--}(\mu)n},\nonumber\\
	\cY_{\mu} &= \prod_{i,j \in \cbc{\pm}} \binom{dn\rho_{ij}(\mu)}{dn\mu_{ij++}, dn\mu_{ij+-}, dn\mu_{ij-+}, dn\mu_{ij--}} ,\nonumber\\
 \cZ_{ \mu} &= \bc{dn\mu_{++--}}! \bc{dn\mu_{-++-}}! \prod_{i \in \cbc{\pm}} \bc{ \bc{dn\mu_{+i-i}}! \bc{dn\mu_{i+i-}}!} \prod_{i,j \in \cbc{\pm}} \bc{dn \mu_{ijij}-1}!!.\nonumber
\end{align}
Indeed, the first factor $\cX_\mu$ counts all pairs $(\sigma,\sigma')$ with $\rho(\sigma,\sigma')=\rho(\mu)$.
Moreover, $\cY_\mu$ accounts for the number of ways of selecting among clones of vertices with a given spin combination those that will be matched to vertices with another specific sign combination.
Finally, $\cZ_{\mu}$ equals the number of possible matchings of clones in accordance with their designations.

Combining \eqref{eqLemma_Philipp0} and \eqref{eqLemma_Philipp1}, we obtain
\begin{align}
	\sum_{\sigma,\sigma'\in\{\pm1\}^{V_n}}&\vecone\cbc{|\sigma\cdot\vecone|,|\sigma'\cdot\vecone|,|\sigma\cdot\sigma'|\leq\delta n}\ex[\exp(-\b\cH_{\G}(\sigma)-\b\cH_{\G}(\sigma'))]\nonumber\\
										  &=\sum_{\mu\in\cM_n}\sum_{\sigma,\sigma'\in\{\pm1\}^{V_n}}\pr\brk{\vec\mu(\sigma,\sigma')=\mu}\exp(-\b\cH(\mu))\nonumber\\
										  &\leq\abs{\cM_n}\max_{\mu\in\cM_n}\sum_{\sigma,\sigma'\in\{\pm1\}^{V_n}}\pr\brk{\vec\mu(\sigma,\sigma')=\mu}\exp(-\b\cH(\mu))\nonumber\\
										  &\leq\exp(O(\log n))\max_{\mu\in\cM_n}\frac{\cX_{\mu}\cY_{\mu}\cZ_{\mu}}{(dn-1)!!}\exp(-\b\cH(\mu)),\label{eqLemma_Philipp2}
\end{align}
because, naturally, $|\cM_n|\leq n^{16}$.
Further, Stirling's formula \eqref{eqStirling} and the explicit formula \eqref{eqdouble} for the double factorial yield the approximations
\begin{align*}
	\log \cX_{\rho} &= n H(\rho(\mu)) + O \bc{\log n},& \log \cY_{\rho, \mu} &= dn (H(\mu)-H(\rho))+ O \bc{\log n},& \log \frac{\cZ_{\rho, \mu}}{(dn-1)!!}&=- \frac{dn}{2} H(\mu)+ O(\log n).
\end{align*}
Combining these formulas and recalling the the Kullback-Leibler divergence, we obtain
\begin{align}\label{eqLemma_Philipp3}
	\max_{\mu\in\cM^\tensor}	\frac{\cX_{\mu}\cY_{\mu}\cZ_{\mu}}{(dn-1)!!}\eul^{-\b\cH(\mu)}&=
	\exp\bc{n\max_{\mu\in\cM^\tensor}\cbc{H(\rho(\mu))-d\KL{\mu}{\rho(\mu)\tensor\rho(\mu)}/2-d\b\cH(\mu)/2}+O(\log n)}.
\end{align}

To simplify the last optimisation problem we reparametrise the exponent in terms of $\alpha\in[-1,1]$.
Combinatorially the optimal choice of $\alpha$ will correspond to the overlap value $\sigma\cdot\sigma'/n$ that renders the largest contribution to the l.h.s.\ of \eqref{eqLemma_Philipp}.
Hence, let $\cM^\tensor(\alpha)$ be the set of all $\mu\in\cM^\tensor$ such that 
\begin{align}\label{eqLemma_Philipp5}
	\rho_{++}(\mu)=\rho_{--}(\mu)&=\frac{1+\alpha}4,&\rho_{+-}(\mu)=\rho_{-+}(\mu)&=\frac{1-\alpha}4.
\end{align}
Then we claim that for any $\alpha\in(-1,1)$,
\begin{align}\label{eqLemma_Philipp4}
	\max_{\mu\in\cM^\tensor(\alpha)}H(\rho(\mu))-d\KL{\mu}{\rho(\mu)\tensor\rho(\mu)}/2-d\b\cH(\mu)/2&\leq f_d(\alpha,\b).
\end{align}
To see this, we first notice that the constrained optimization problem on the l.h.s of \ref{eqLemma_Philipp4} is upper bounded by the result of the unconstrained optimization problem, i.e.
\begin{align}
    &\max_{\mu\in\cM^\tensor(\alpha)}H(\rho(\mu))-d\KL{\mu}{\rho(\mu)\tensor\rho(\mu)}/2-d\b\cH(\mu)/2 \notag \\ &\qquad \qquad \qquad \leq \max_{\mu\in M(\alpha)}H(\rho(\mu))-d\KL{\mu}{\rho(\mu)\tensor\rho(\mu)}/2-d\b\cH(\mu)/2 \label{eqLemma_Philipp4_1} 
\end{align}
where $M(\alpha)$ is the set of all probability measures on $\cP_\alpha(\cbc{\pm 1}^4)$ parametrised by $\alpha$.
Moreover, we notice that the function $\mu\in M(\alpha)\mapsto H(\rho(\mu))-d\KL{\mu}{\rho(\mu)\tensor\rho(\mu)}/2-d\b\cH(\mu)/2$ is concave because $H(\rho(\mu))$ is constant on $M(\alpha)$, the Kullback-Leibler divergence is strictly convex and the function $\cH(\mu)$ is linear.
Hence, it suffices to find the (unique) zero of the derivative of $\KL{\mu}{\rho(\mu)\tensor\rho(\mu)}/2+\b\cH(\mu)/2$ subject to \eqref{eqLemma_Philipp5}.

Letting $z_\a=(1+\ebb)(1+\a^2)/4+\eb(1-\a^2)/2$, we claim that $\mu_\a$ given by
\begin{align}
	\mu_{\a,++++}&=\mu_{\a,----}=\frac{(1+\a)^2}{16z_{\a}}\ebb,\qquad
	\mu_{\a,+-+-}=\mu_{\a,-+-+}=\frac{(1-\a)^2}{16z_{\a}}\ebb,\label{eqLemma_Philipp6}\\
\mu_{\a,+++-}&=\mu_{\a,++-+}=\mu_{\a,+-++}=\mu_{\a,-+++}=
\mu_{\a,---+}=\mu_{\a,--+-}=\mu_{\a,-+--}=\mu_{\a,+---}=\frac{1-\a^2}{16z_{\a}}\eb,\label{eqLemma_Philipp7}\\
\mu_{\a,++--}&=\mu_{\a,--++}=\frac{(1+\a)^2}{16z_{\a}},\qquad
\mu_{\a,+--+}=\mu_{\a,-++-}=\frac{(1-\a)^2}{16z_{\a}},\label{eqLemma_Philipp8}
\end{align}
fits the bill.
Indeed, we calculate
\begin{align}\label{eqLemma_Philipp9}
\mu_{\a,++++}\log\frac{\mu_{\a,++++}}{\rho_{\a,++}\rho_{\a,++}}
&=\mu_{\a,++++}\log\frac{\ebb}{z_\a}=-\mu_{\a,++++}(\log(z_\a)+2\b).
\end{align}
The same formula holds with $++++$ replaced by any of the other sign patterns from \eqref{eqLemma_Philipp6}, i.e., $----$, $+-+-$, $-+-+$.
Moreover,
\begin{align}\label{eqLemma_Philipp10}
\mu_{\a,+++-}\log\frac{\mu_{\a,+++-}}{\rho_{\a,++}\rho_{\a,+-}}
&=\mu_{\a,+++-}\log\frac{\eb}{z_\a}=-\mu_{\a,+++-}(\log(z_\a)+\b),
\end{align}
and similarly for the other seven patterns from \eqref{eqLemma_Philipp7}.
Further,
\begin{align}\label{eqLemma_Philipp11}
\mu_{\a,++--}\log\frac{\mu_{\a,++--}}{\rho_{\a,++}\rho_{\a,--}}&=-\mu_{\a,++--}\log z_\a
\end{align}
and analogously for the other sign patterns from \eqref{eqLemma_Philipp8}.
Consequently, the derivatives work out to be
\begin{align}
	\pd{\mu_{\alpha, ++++}}\bc{\KL{\mu}{\rho(\mu)\tensor\rho(\mu)}+\b\cH(\mu)}\bigg|_{\mu_\a} &=1+\log\frac{\mu_{\a,++++}}{\rho_{++}(\mu_{\a})^2}+2\b=1-\log(z_\a),\label{eqLemma_Philipp12}\\
	\pd{\mu_{\alpha, +++-}}\bc{\KL{\mu}{\rho(\mu)\tensor\rho(\mu)}+\b\cH(\mu)}\bigg|_{\mu_\a} &=1+\log\frac{\mu_{\a,+++-}}{\rho_{++}(\mu_{\a}) \rho_{+-}(\mu_{\a})}+\b=1-\log(z_\a),\label{eqLemma_Philipp13}\\
	\pd{\mu_{\alpha, ++--}}\bc{\KL{\mu}{\rho(\mu)\tensor\rho(\mu)}+\b\cH(\mu)}\bigg|_{\mu_\a} &=1+\log\frac{\mu_{\a,++--}}{\rho_{++}(\mu_{\a}) \rho_{--}(\mu_{\a})}=1-\log(z_\a)\label{eqLemma_Philipp14}.
\end{align}
In each case the same calculation applies to the other sign patterns from the respective line \eqref{eqLemma_Philipp6}--\eqref{eqLemma_Philipp8}.
Since the right hand sides of \eqref{eqLemma_Philipp12}--\eqref{eqLemma_Philipp14} are identical, the constraint that $\mu$ belongs to the simplex $\cP(\{\pm1\}^4)$ shows that $\mu_\a$ is a stationary point and therefore the unique maximiser of $\KL{\mu}{\rho(\mu)\tensor\rho(\mu)}/2+\b\cH(\mu)/2$.
Moreover, combining \eqref{eqLemma_Philipp9}--\eqref{eqLemma_Philipp11}, we find
\begin{align}\label{eqLemma_Philipp_comb}
-\KLmrra-\b\cH(\mu_\a)&=
\log(z_\a)=\log\bc{1+\a^2\bcfr{1-\eb}{1+\eb}^2}+2\log\frac{1+\eb}{2},
\end{align}
Thus, \eqref{eqLemma_Philipp4} follows from the definition \eqref{eqTech5a} of $f_d(\a,\b)$ and \eqref{eqLemma_Philipp_comb}.

To complete the proof consider any $\mu\in M(\alpha)$.
Then \eqref{eqLemma_Philipp00} ensures that there exists $\mu'\in M(0)$ such that $\|\mu-\mu'\|_2=O(\delta)$.
Consequently, since the function $\mu\mapsto H(\rho(\mu))-d\KL{\mu}{\rho(\mu)\tensor\rho(\mu)}/2-d\b\cH(\mu)/2$ is differentiable, the bound \eqref{eqLemma_Philipp4} shows that
\begin{align}\label{eqLemma_Philipp90}
	\max_{\mu\in M(\alpha)}H(\rho(\mu))-d\KL{\mu}{\rho(\mu)\tensor\rho(\mu)}/2-d\b\cH(\mu)/2&\leq f_d(0,\b)+O(\delta).
\end{align}
Thus, the assertion follows from \eqref{eqLemma_Philipp2} and \eqref{eqLemma_Philipp90}.
\end{proof}

\begin{corollary}\label{Cor_Philipp}
For any $d\geq3,\b>0$ we have $$\ex\brk{Z_{\G(n,d),\b}^2\vecone\cbc{\cO}}\leq \ex\brk{Z_{\G(n,d),\b}}^2\exp(o(n)).$$
\end{corollary}
\begin{proof}
This is an immediate consequence of \Cor~\ref{Lemma_sm_trunc} and \Lem~\ref{Lemma_Philipp}.
\end{proof}

\begin{proof}[Proof of \Lem~\ref{Lemma_O2}]
The lemma follows from \Lem~\ref{Lemma_mean}, \Lem~\ref{Cor_simple*} and \Cor~\ref{Cor_Philipp}.
\end{proof}

\begin{proof}[Proof of \Prop~\ref{prop_recon_energy}]
The proposition follows from
\Lem~\ref{Cor_simple*},
\Lem~\ref{Lemma_hunch},
\Lem~\ref{cor_reconstruction},
\Lem~\ref{Lemma_mean},
and \Cor~\ref{Cor_Philipp}.
\end{proof}

\section{Proof of \Prop~\ref{prop_Taylor}} \label{sec_thm_2}

\noindent
We begin by deriving an explicit formula for $\cB(\pi_\eps^*, \beta)$.
Recall that $\Lambda(x)=x\log x$.

\begin{lemma}\label{Lemma_BIsing}
	Let $d\geq3,\b>0$.
	Then for small enough $\eps>0$ we have
	\begin{align*}
		\cB_{\text{Ising}}(\pi_\eps^*, \beta, d) &= \frac{\sum_{i=1}^d \binom{d}{i} 2^{-d} \Lambda\bc{\sum_{\sigma =\pm 1}\bc{1-\bc{1-\eb}\bc{\frac 12 + \sigma \eps}}^i \bc{1-\bc{1-\eb}\bc{\frac 12 - \sigma \eps}}^{d-i}}}{2 \bc{(1+\eb)/2}^d} \\
												 & \qquad \qquad - \frac{d \bc{\Lambda \bc{1-\bc{1-\eb}\bc{\frac 12 +2 \eps^2}} + \Lambda \bc{1-\bc{1-\eb} \bc{\frac 12 - 2 \eps^2} }}}{2(1+\eb)}
	\end{align*}
\end{lemma}

\begin{proof}
	The expression follows straight from plugging the distribution $\pi_\eps^*$ from \eqref{eq_distribution} into the Bethe functional from \eqref{eq_Bethe_exp}. Let us shed light on its combinatorial meaning. The first term represents the 'weighted penalty factor' arising at a root vertex with $d$ adjacent vertices. Since we polarise each of these adjacent vertices with probability $1/2$ independently, the number of adjacent vertices polarised to $\eps$ and $-\eps$ follows a binomial distribution. The term $\binom{d}{i} 2^{-d}$ captures the corresponding probability while the term inside $\Lambda(\cdot)$ describes the resulting penalty factor over all adjacent vertices summed over the $+1$ and $-1$ spins at the root vertex. The second term represents the 'weighted penalty factor' between two vertices connected via an edge. Here, the first corresponds to the case that both vertices are polarised to the same spin, while the second summand picks up the penalty factor for polarisation towards different spins.
\end{proof}

\begin{lemma}\label{Lemma_Btaylor}
	Let $d\geq3,\b>0$.
	Then as $\eps\to0$,
\begin{align*}
		\cB_{\text{Ising}}(\pi_\eps^*,\beta,d) &= \log 2+\frac{d}{2}\log\frac{1+\eul^{-\b}}{2}+
		\frac{4\eps^4 d \eul^{-2\b} \bc{\eul^\beta -1}^2 \bc{\eul^{2\b} (d-2) - 2d \eul^\beta + d-2}}{(1+\eul^\beta)^2(1+\eul^{-\beta})^2}+O(\eps^5).
	\end{align*}
\end{lemma}
\begin{proof}
	\Lem~\ref{Lemma_BIsing} shows that the function $\eps\mapsto\cB_{\text{Ising}}(\pi_\eps^*,\beta,d)$ has five continuous derivatives in for small enough $\eps>0$.
Hence, Taylor's formula yields
\begin{align}\nonumber
    \cB_{\text{Ising}}(\pi_\eps^*, \beta, d) = &\cB_{\text{Ising}}(\pi_0^*, \beta) + \eps \frac{\partial}{\partial \eps} \cB_{\text{Ising}}(\pi_\eps^*, \beta, d) \vert_{\eps=0} + \frac{\eps^2}{2} \frac{\partial^2}{\partial \eps^2} \cB_{\text{Ising}}(\pi_\eps^*, \beta, d) \vert_{\eps=0} \\
    &\qquad \qquad + \frac{\eps^3}{6} \frac{\partial^3}{\partial \eps^3} \cB_{\text{Ising}}(\pi_\eps^*, \beta, d) \vert_{\eps=0} + \frac{\eps^4}{24} \frac{\partial^4}{\partial \eps^4} \cB_{\text{Ising}}(\pi_\eps^*, \beta, d) \vert_{\eps=0} + O \bc{\eps^5}.\label{eqLemma_Btaylor1}
\end{align}
The formula for $\cB_{\text{Ising}}(\pi_\eps^*,\beta,d)$ from \Lem~\ref{Lemma_BIsing} is complicated but explicit.
Therefore, we can rely on a computer algebra system to calculate the first four derivatives of $\cB_{\text{Ising}}(\pi_\eps^*,\beta,d)$ symbolically at $\eps=0$.
The result of this calculation reads
\begin{align}\label{eqLemma_Btaylor2}
   \frac{\partial}{\partial \eps} \cB_{\text{Ising}}(\pi^*, \beta, d) \vert_{\eps=0} &= 
   \frac{\partial^2}{\partial \eps^2} \cB_{\text{Ising}}(\pi^*, \beta, d) \vert_{\eps=0} =
   \frac{\partial^3}{\partial \eps^3} \cB_{\text{Ising}}(\pi^*, \beta, d) \vert_{\eps=0} = 0, \\
   \frac{\partial^4}{\partial \eps^4} \cB_{\text{Ising}}(\pi^*, \beta, d) \vert_{\eps=0} &= \frac{96 d e^{-2\b} \bc{e^\beta -1}^2 \bc{e^{2\b} (d-2) - 2d e^\beta + d-2}}{(1+e^\beta)^2(1+e^{-\beta})^2}
   \label{eqLemma_Btaylor3}
\end{align}
Plugging \eqref{eqLemma_Btaylor2}--\eqref{eqLemma_Btaylor3} into \eqref{eqLemma_Btaylor1} yields the desired formula.
\end{proof}

\begin{proof}[Proof of \Prop~\ref{prop_Taylor}]
	Let $d\geq3$.
	A few lines of algebra reveal that $\eul^{2\b} (d-2) - 2d \eul^\beta + d-2>0$ if $\b>\b^*(d)$.
	Therefore, \Lem~\ref{Lemma_Btaylor} shows that for small enough $\eps>0$ and $\b>\b^*(d)$ we have $\cB_{\text{Ising}}(\pi_\eps^*,\beta,d) > \log 2+\frac{d}{2}\log\frac{1+\eul^{-\b}}{2}$, as claimed.
\end{proof}

\section{Proof of \Prop~\ref{Prop_Explicit_Yprime}} \label{sec_cor}
\noindent
We treat $X_1,X_2$ from \Lem~\ref{Prop_PD} separately.
Let us begin with $X_2$, which is the easier case.
In the following, fix any $d\geq3$, $\alpha\in(0,1/2)$ and $z\in(0,1)$ and set $y=y(\b)=-\b^{-1}\log z$.

\begin{lemma}\label{Lem_Explicit_Yprime}
	We have $ \limb \log \Erw[X_2^{y}] = \log \bc{1-2\a^2+2\a^2z}. $
\end{lemma}
\begin{proof}
	Recalling the definition of $\fr_\alpha$ from \eqref{eqrhoalpha} and writing the expectation out explicitly, we obtain
\begin{align}
	\ex[X_2^y]=\ex\brk{\bc{1-(1-\eul^{-\b}) \sum_{\tau\in\{\pm 1\}} \RHO_{1}(\tau)\RHO_{2}(\tau)}^y}&=
	\sum_{r_1,r_2 \in \{0,1,-1\}} \fr_\a(r_{1})\fr_\a(r_{2}) \brk{ 1- \frac{(1-\eb)(1+r_1r_2)}2 }^y.
	\label{eqLem_Explicit_Yprime1}
\end{align}
To evaluate this expression we consider the possible values of the product $r_1r_2$.
\begin{description}
	\item[Case 1: $r_1r_2=-1$] by the choice \eqref{eqrhoalpha} of $\fr$ this event has probability $2\a^2$ and
		\begin{align}	\label{eqLem_Explicit_Yprime2}
			1- \frac{(1-\eb)(1+r_1r_2)}2 =1.
		\end{align}
	\item[Case 2: $r_1r_2=1$] in this case, which occurs with probability $2\a^2$ as well, we obtain
\begin{align}\label{eqLem_Explicit_Yprime3}
	1- \frac{(1-\eb)(1+r_1r_2)}2 =\eb.
		\end{align}
	\item[Case 3: $r_1r_2=0$] naturally this event occurs with the remaining probability $1-4\a^2$ and
\begin{align}\label{eqLem_Explicit_Yprime4}
	1- \frac{(1-\eb)(1+r_1r_2)}2 =\frac{1+\eb}{2}.
		\end{align}
\end{description}
Combining \eqref{eqLem_Explicit_Yprime1}--\eqref{eqLem_Explicit_Yprime4}, we obtain
\begin{align}
	\Erw\brk{\bc{1-(1-\eul^{-\b}) \sum_{\tau\in\{\pm 1\}} \RHO_{1}(\tau)\RHO_{2}(\tau)}^y}&= 2\a^2+2\a^2\eul^{-\b y}+(1-4\a^2)\bcfr{1+\eb}{2}^y.
	\label{eqLem_Explicit_Yprime5}
\end{align}
Finally, since $z=\exp(-\b y)$ and 
\begin{align*}
	\bcfr{1+\eb}{2}^y=\exp\bc{-\frac{\log z}{\b}\log(1+\eul^{-b})}\to1\qquad\mbox{as $\b\to\infty$,}
\end{align*}
combining \eqref{eqLem_Explicit_Yprime1} and \eqref{eqLem_Explicit_Yprime5} shows that $ \limb\log \ex[X_2^y]=\log(1-2\a^2+2\a^2z), $ as desired.
\end{proof}

The computation of $\Erw \brk{X_1^{y}}$ is a little more intricate.
Combinatorially speaking, the basic idea is this.
Consider the picture on the left of Figure~\ref{Fig_interpolation}.
The expression $$X_1=\sum_{\tau\in\{\pm 1\}} \prod_{h=1}^{d}1-(1-\eul^{-\b})\RHO_{h}(\tau)$$ represents the contribution to the partition function of a single white vertex along with its adjacent blue boxes.
Each of these boxed represents an imaginary vertex, or a `field' in physics jargon, that takes the spin $\tau$ with probability $\RHO_h(\tau)$.
The spins of these imaginary vertices are mutually independent.
Hence, the sum on $\tau$ in the definition of $X_1$ accounts for the two possible choices of spin for the white vertex.
Thus, if we let $\vH(\tau)$ be the number of imaginary vertices with spin $\tau$, then the product $\prod_{h=1}^{d}1-(1-\eul^{-\b})\RHO_{h}(\tau)$ equals the expected Boltzmann weight
\begin{align*}
	\ex[\exp(-\b\vH(\tau))\mid\RHO_1,\ldots,\RHO_d].
\end{align*}
Furthermore, the fields $\RHO_h(\tau)$ can be either `soft', i.e., $\RHO_h(\tau)=1/2$, or `hard', meaning $\RHO_h(\tau)\in\{0,1\}$.
As in the proof of \Lem~\ref{Lem_Explicit_Yprime} we will see that in the limit $\b\to\infty$ and $y\to0$ the soft fields are inconsequential.
In effect, the computation of $X_1$ will come down to studying the random variable $\sum_{\tau\in\{\pm1\}}\tau\vecone\{\RHO_h(\tau)=1\}$, which gauges the relative strength of the hard fields.
In other words, the calculation of $\Erw\brk{X_1^{y}}$ comes down to analysing a random walk.
Let us get down to the details.

\begin{lemma}\label{Lem_Explicit_Y}
We have $ \limb \log \Erw\brk{X_1^{y}} =  \log \bc{\zeta\cAda^d \xi}.  $
\end{lemma}
\begin{proof}
	Letting
\begin{align*}
	\vR_{-1}& = \sum_{h=1}^{d} \vecone \cbc{\RHO_{h}(1)=0},&
	\vR_{0}& = \sum_{h=1}^{d} \vecone \cbc{\RHO_{h}(1)=1/2},&
	\vR_1 &= \sum_{h=1}^{d} \vecone \cbc{\RHO_{h}(1)=1}
\end{align*}
we can write the random walk as $\sum_{\tau\in\{\pm1\}}\tau\vecone\{\RHO_h(\tau)=1\}=\vR_1-\vR_{-1}$.
Hence,
\begin{align}\label{eqLem_Explicit_Y1}
     \prod_{h=1}^{d} \bc{ 1-(1-\eb ) \RHO_{h}(1) }^y   &= \exp(-\b y\vR_1)\bcfr{1+\eb }2^{y\vR_0},\\
    \prod_{h=1}^{d} \bc{ 1-(1-\eb ) \RHO_{h}(-1) }^y &= \exp(-\b y\vR_{-1})\bcfr{1+\eb}2^{y\vR_0}.
\label{eqLem_Explicit_Y2}
\end{align}
Since $y=-\b^{-1}\log z$, for any non-negative integers $R_1,R_{-1},R_0\geq0$ such that $R_1+R_0+R_{-1}=d$ we have
\begin{align} \label{eqLem_Explicit_Y3}
	\limb(\exp(-\b yR_1)+\exp(-\b yR_{-1}))\bcfr{1+\eb }2^{yR_0}&=z^{R_1\wedge R_{-1}}.
\end{align}
Thus, combining \eqref{eqLem_Explicit_Y1}--\eqref{eqLem_Explicit_Y3}, we obtain
\begin{align} \label{eqLem_Explicit_Y4}
	\limb \log \Erw[X_1^{y}] = \log \Erw\brk{z^{\vR_1\wedge\vR_{-1}}}.
\end{align}

To calculate the mean on the r.h.s.\ consider a $d$-step symmetric random walk on $\{0,1,\ldots,d\}$ with a reflective barrier at $0$.
The walk starts at $0$ and the available moves are $+1$, $-1$ or $0$, with probabilities $\a$, $\a$ and $1-2\a$ respectively.
We couple this random walk with the probability space $(\RHO_1,\ldots,\RHO_d)$ such that $\vR_1$ and $\vR_{-1}$ count the $\pm1$ moves of the random walk, respectively.
Thus, $\vR_0=d-\vR_1-\vR_{-1}$ equals the number of $0$-moves and $\abs{\vR_1-\vR_{-1}}$ is the final position of the walk.
To study this random walk we remember the matrix $\cM$ from \eqref{eq_Mdef} and introduce
\begin{align*}
	\fA=(1-2\a) \id + 2 \a t \cM,
\end{align*}
where $t$ is a formal variable that we introduce to track the walk's movements.
Specifically, for any $i\in[d]$ the $(1,i)$-entry of the $d$-th power of $\fA$ works out to be
\begin{align} \label{eq_random_walk}
	\fA^d_{1\,i} = \sum_{k=i-1}^d t^k \Pr \brk{ \vR_1 + \vR_{-1} = k \mbox{ and } \abs{\vR_1 - \vR_{-1}} = i-1}
\end{align}
Finally, we introduce the vector $$\fx = ( 1, t^{-1}, t^{-2}, t^{-3}, \cdots )^T\in\RR^{(d+1)\times 1}.$$
Then recalling the definition of vector $\zeta$ from \eqref{eq_v_def}, we readily find
\begin{align} \label{eq_interpol_aux}
    \zeta \fA^d \fx = \sum_{k=0}^d t^k \Pr \brk{\vR_1 + \vR_{-1} - \abs{\vR_1 - \vR_{-1}} = k} = \Erw \brk{t^{\vR_1 + \vR_{-1} - \abs{\vR_1 - \vR_{-1}}}}.
\end{align}
Let us shed light on the combinatorial meaning of \eqref{eq_interpol_aux}. $\zeta \fA^d$ is a $(d+1)$-dimensional vector where the $i$th entry captures the probability of all random walks that end up at position $i-1$ and where the exponent of $t$ measures the number of non-stationary steps performed to reach position $i-1$. Thus, using the definition from \eqref{eq_random_walk} we have
\begin{align*}
    \zeta \fA^d = \bc{\fA^d_{1\,1}, \fA^d_{1\,2}, \fA^d_{1\,3}, \dots}.
\end{align*}
The multiplication with vector $\fx$ then deducts the final position $\abs{\vR_1-\vR_{-1}}$ of the random walk from the exponent of $t$. In effect, the exponent now captures the total number of offsetting non-stationary steps of the random walk which is precisely twice the minimum of $\vR_1$ and $\vR_{-1}$. This relationship can be compactly written in the basic identity
\begin{align*}
    \vR_1 + \vR_{-1} - \abs{\vR_1 - \vR_{-1}} = 2 \bc{\vR_1\wedge\vR_{-1}}.
\end{align*}
We are now in a position to relate \eqref{eqLem_Explicit_Y3} to \eqref{eq_interpol_aux} by writing
\begin{align*}
    \limb \log \Erw[X_1^{y}] = \log \Erw\brk{z^{\vR_1\wedge\vR_{-1}}} = \Erw \brk{\sqrt{z}^{\vR_1 + \vR_{-1} - \abs{\vR_1 - \vR_{-1}}}} = \log \bc{\zeta \fA^d \fx \vert_{t=\sqrt{z}}}.
\end{align*}
Since $\cAda$ from \eqref{eq_A_def} and $\xi$ from \eqref{eq_w_def} were defined in terms of $\sqrt{z}$ rather than $t$ we conclude that
\begin{align*}
    \limb \log \Erw\brk{X_1^{y}} = \log \bc{\zeta \fA^d \fx \vert_{t=\sqrt{z}}} = \log \bc{\zeta\cAda^d \xi}
\end{align*}
as claimed.
\end{proof}

\begin{proof}[Proof of \Prop~\ref{Prop_Explicit_Yprime}]
The proposition is an immediate consequence of \Lem s~\ref{Lem_Explicit_Yprime} and~\ref{Lem_Explicit_Y}.
\end{proof}


\end{document}